\documentclass[12pt]{amsproc}  
\usepackage{amsmath,amsthm,amssymb,amsfonts,eucal,fullpage,latexsym,mathrsfs,stmaryrd}
\usepackage{enumerate}
\usepackage{color}
\usepackage{lipsum}
\usepackage{array}
\usepackage{lineno}
\usepackage{graphicx}
\theoremstyle{plain}

\numberwithin{equation}{section}
\theoremstyle{plain}
\newtheorem{Proposition}[equation]{Proposition}
\newtheorem{Corollary}[equation]{Corollary}
\newtheorem*{Corollary*}{Corollary}
\newtheorem{Theorem}[equation]{Theorem}
\newtheorem*{Theorem*}{Theorem}
\newtheorem{Lemma}[equation]{Lemma}
\theoremstyle{definition}
\newtheorem{Definition}[equation]{Definition}

\newtheorem{Remark}[equation]{Remark}

\def\phi{\varphi}

\renewcommand{\leq}{\leqslant}
\renewcommand{\geq}{\geqslant}
\renewcommand{\subset}{\subseteq}
\renewcommand{\supset}{\supseteq}

\begin{document}



\title{On Interpolation by Functions in $\ell_A^p$}
\author[Cheng]{Raymond Cheng}
\address{Department of Mathematics and Statistics, Old Dominion University, Norfolk, VA 23529, USA. } \email{rcheng@odu.edu}
\author[Felder]{Christopher Felder}
\address{Department of Mathematics, Indiana University, Bloomington, IN 47405, USA. } \email{cfelder@iu.edu}
\subjclass[2020]{Primary 46E15; Secondary 30J99, 30H99.}

\date{\today}

\begin{abstract}
This work explores several aspects of interpolating sequences for $\ell^p_A$, the space of analytic functions on the unit disk with $p$-summable Maclaurin coefficients. Much of this work is communicated through a Carlesonian lens.
 We investigate various analogues of Gramian matrices, for which we show boundedness conditions are necessary and sufficient for interpolation, including a characterization of universal interpolating sequences in terms of Riesz systems. We also discuss weak separation, giving a characterization of such sequences using a generalization of the pseudohyperbolic metric. Lastly, we consider Carleson measures and embeddings. \end{abstract}

\maketitle


\section{Introduction}

Let $\{z_k\}_{k=0}^{\infty}$ be a sequence of distinct points in the open unit disk $\mathbb{D}$ of the complex plane, and let $\{w_k\}_{k=0}^{\infty}$ be a sequence of complex numbers.  An interpolation problem, broadly speaking, is to find an analytic function $f$ on $\mathbb{D}$ such that
\[
        f(z_k) = w_k
\]
for all $k =0, 1, 2,\ldots$.  Furthermore, it is natural to seek a characterization of those sequences $\{z_k\}$ and $\{w_k\}$ for which such a function $f$ always exists, and belongs to a certain class of analytic functions on $\mathbb{D}$.  As explained by Duren \cite[p.~148]{Dur}, intuition suggests that the points $\{z_k\}$  must not lie too ``close together,'' lest a highly oscillatory choice of targets $\{w_k\}$ fail to be interpolated by a function of the prescribed class.  

This notion is illustrated by a theorem of Carleson \cite{Carl}, which characterizes interpolation by functions in $H^{\infty}$.   A sequence $\{z_k\}$ in $\mathbb{D}$ is said to be {\it uniformly separated} if there exists $\delta>0$ such that 
\[
       \prod_{j=0,\,j\neq k}^{\infty} \Big| \frac{z_k-z_j}{1 - \bar{z}_j z_k}\Big| \geq \delta
\]
for all $k = 1, 2, 3,\ldots$.
\begin{Theorem}[Carleson]
   Let $\{z_k\}$ be a sequence of points in $\mathbb{D}$.  Then $\{z_k\}$ has the property that
   for any bounded sequence $\{w_k\}$ of  complex numbers, there exists $f \in H^{\infty}$ such that
   \[
           f(z_k) = w_k
   \]
   for all $k = 1, 2, 3,\ldots$, if and only if $\{z_k\}$ is uniformly separated.
\end{Theorem}

This was extended to $H^p$ by Shapiro and Shields \cite{SS}.  See the books by Seip \cite{Seip}, and by Agler and McCarthy \cite{AM}, for a modern exposition of interpolation in numerous other spaces of functions.

The present paper is concerned with interpolation by functions in $\ell^p_A$, the space of analytic functions in $\mathbb{D}$ whose Maclaurin coefficients are $p$-summable, with $1<p<\infty$.   Vinogradov \cite{Vino3, Vino4} (for English translations, with Khavin, see \cite{VK1,VK2})  derived exact conditions for interpolation in $\ell^p_A$ provided that the sequence $\{z_k\}$ lies in a Stoltz domain, or that it tends rapidly enough to the boundary.   Our approach connects interpolation in $\ell^p_A$ with associated sequences of functionals, infinite matrices, a nonlinear functional equation, notions of separation, and Carleson measures.
In particular, after providing some background information on the $\ell^p_A$ spaces in the next section, we move to Section \ref{bnd-cond}, which is concerned with various upper and lower bounds that characterize universal interpolating sequences. Section \ref{cont} involves studying limits of truncated interpolation problems to deduce an interpolation result based on the geometry of the Banach space in terms of a minimality condition.  Criteria for interpolation are expressed in terms of matrix conditions in Section \ref{matrix-analysis}.  There arises a pair of nonlinear operators that extend the notion of a Gramian matrix to the case $p\neq 2$.
Section \ref{weak-separation} characterizes sequences which are weakly separated by the multiplier algebra of $\ell^p_A$. We close with a section on Carleson measures for $\ell^p_A$, and a handful of open questions.


\section{Preliminaries}\label{prelim}

For $0 < p \leq \infty$, the space $\ell^p_A$ is defined to be space of analytic functions on the open unit disk $\mathbb{D}$ for which the Maclaurin coefficients are $p$th order summable, i.e.,
\[
\ell^p_A := \left\{ f(z) = \sum_{n\ge0}a_nz^n \in \operatorname{Hol}(\mathbb{D}) :  \sum_{n\ge0}|a_n|^p < \infty \right\}.
\]

This function space is endowed with the norm (or quasinorm, if $0<p\leq 1$) that it inherits from the sequence space $\ell^p$.  Thus let us write 
\[
       \|f\|_p  = \|\{a_k\}_{k=0}^{\infty}\|_{\ell^p}
\]
for any
$
      f(z) = \sum_{k=0}^{\infty} a_k z^k
$
belonging to $\ell^p_A$.  When $p=2$, we recover the classical Hardy space $H^2$ on the disk, however, we emphasize that $\|\cdot\|_p$ refers to the norm on $\ell^p_A$, and not the norm on the Hardy space $H^p$, or some other function space parametrized by $p$.

We limit our attention to the range $1<p<\infty$.  For then $\ell^p_A$ is reflexive, smooth and uniformly convex.  In particular, it enjoys the unique nearest point property, and each nonzero vector has a unique norming functional.

Throughout this paper,  $q$ will be the H\"{o}lder conjugate to $p$, that is, $1/p + 1/q = 1$ holds.  
We recall that for $1 \leq p < \infty$, $p\neq 2$, the dual space of $\ell^p_A$ can be identified with $\ell^{q}_A$, under the pairing
\[
       \langle f, g\rangle = \sum_{k=0}^{\infty} f_k g_k,
\]
where $f(z) = \sum_{k=0}^{\infty} f_k z^k \in \ell^p_A$ and $g(z) = \sum_{k=0}^{\infty} g_k z^k\in \ell^q_A$.

Point evaluation at any $w \in \mathbb{D}$ is a bounded linear functional  on $\ell^p_A$.  It is implemented by the kernel function $\Lambda_w \in \ell^q_A$ given by
\[
        \Lambda_w(z) := \sum_{k=0}^{\infty} w^k z^k, \ \ \ z \in \mathbb{D}.
\]
Indeed, for any $f(z) = \sum_{k=0}^{\infty} a_k z^k \in \ell^p_A$ we have
\[
      f(w) = \langle f, \Lambda_w \rangle = \sum_{k=0}^{\infty} a_k w^k.
\]
The norm of $\Lambda_w$ (as either a vector or a functional) is
\[
     \|\Lambda_w\|_q = \frac{1}{(1 - |w|^q)^{1/q}}.
\]

\bigskip

There is a sensible way to define ``inner function'' in the context of $\ell^p_A$, that employs a notion of orthogonality in general normed linear spaces.  

Let $\mathbf{x}$ and $\mathbf{y}$ be vectors belonging to a normed linear space $\mathscr{X}$.  We say that $\mathbf{x}$ is {\em orthogonal} to $\mathbf{y}$ in the Birkhoff-James sense \cite{AMW, Jam}  if
\begin{equation}\label{2837eiywufh[wpofjk}
      \|  \mathbf{x} + \beta \mathbf{y} \|_{\mathscr{X}} \geq \|\mathbf{x}\|_{\mathscr{X}}
\end{equation}
for all scalars $\beta$, 
and in this case we write $\mathbf{x} \perp_{\mathscr{X}} \mathbf{y}$.
Birkhoff-James orthogonality is one way to extend the concept of orthogonality from an inner product space to normed spaces.   There are other ways to generalize orthogonality, but this approach is particularly useful since it is relates directly to an extremal condition, namely \eqref{2837eiywufh[wpofjk}.

If $\mathscr{X}$ is a Hilbert space, then the usual orthogonality relation $\mathbf{x} \perp \mathbf{y}$ is equivalent to $\mathbf{x} \perp_{\mathscr{X}} \mathbf{y}$.  More generally, however, the relation $\perp_{\mathscr{X}}$ is neither symmetric nor linear.  When $\mathscr{X} = \ell^{p}_A$,  let us write $\perp_p$ instead of  $\perp_{\ell^p_A}$.  
There is a practical criterion for the relation $\perp_p$ when $1<p<\infty$.

\begin{Theorem}[James \cite{Jam}]
Suppose that $1<p<\infty$.  Then for $f(z)= \sum_{k=0}^{\infty} f_k z^k$ and $g(z)= \sum_{k=0}^{\infty} g_k z^k$ belonging to $\ell^p_A$ we have
\begin{equation}\label{BJp}
 {f} \perp_{p} {g}  \iff   \sum_{k=0}^{\infty} |f_k|^{p - 2} \overline{f}_k g_k  = 0,
\end{equation}
where any occurrence of ``$|0|^{p - 2} 0$'' in the right side  is interpreted as zero.
\end{Theorem}

In light of \eqref{BJp} we define, for a complex number $\alpha=re^{i\theta}$, and any $s > 0$, the quantity
\begin{equation}\label{definition-de-zs}
\alpha^{\langle s \rangle} = (re^{i\theta})^{\langle s \rangle} := r^{s} e^{-i\theta}.
\end{equation}
It is readily seen that for any complex numbers $\alpha$ and $\beta$,  exponent $s>0$, and integer $n \geq 0$, we have
\begin{align*}
(\alpha\beta)^{\langle s\rangle} &= \alpha^{\langle s\rangle} \beta^{\langle s\rangle}\\
|\alpha^{\langle s \rangle}| &= |\alpha|^s\\
\alpha^{\langle s \rangle}\alpha &= |\alpha|^{s+1}\\
(\alpha^{\langle s \rangle})^n &= (\alpha^n)^{\langle s \rangle}\\
(\alpha^{\langle p-1 \rangle})^{\langle q-1 \rangle} &= \alpha.
\end{align*}

Further to the notation \eqref{definition-de-zs}, for $f(z) = \sum_{k=0}^{\infty} f_k z^k$, let us write
\begin{equation}\label{uewuweuer}
   f^{\langle s \rangle}(z)  := \sum_{k=0}^{\infty}  f_k^{\langle s \rangle} z^k
\end{equation}
for any $s>0$.  Similarly, for any matrix or vector with complex-valued entries $A = [a_{j,k}]$, we take $A^{\langle s\rangle}$ to mean the matrix $[a_{j,k}^{\langle s\rangle}]$.
By comparing with the case  $p=2$, we can think of taking the  ${\langle s \rangle}$ power as generalizing complex conjugation.

If $f \in \ell^{p}_A$, it is easy to verify that $f^{\langle p - 1\rangle} \in \ell^{q}_A$.   Thus from  \eqref{BJp} we get
\begin{equation}\label{ppsduwebxzrweq5}
f \perp_{p} g \iff \langle g, f^{\langle p - 1\rangle} \rangle = 0.
\end{equation}
Consequently the relation $\perp_{p}$ is linear in its second argument, when $1<p<\infty$, and it then makes sense to speak of a vector being orthogonal to a subspace of $\ell^{p}_A$.  In particular, if $f \perp_p g$ for all $g$ belonging to a subspace $\mathscr{M}$ of $\ell^p_A$, then
\[
       \| f + g \|_p  \geq  \|f\|_p
\]
for all $g \in \mathscr{M}$.  That is, $f$ solves a nearest-point problem in relation to the subspace $\mathscr{M}$.
Direct calculation will also confirm that
\[
     \langle f, f^{\langle p-1\rangle} \rangle = \|f\|^p_p,
\]
and hence $f^{\langle p-1\rangle}/ \|f\|^{p-1}_p \in \ell^q_A$ is the norming functional of $f \in \ell^p_A \setminus \{0\}$ (smoothness ensures that the norming functional is unique).
With this concept of orthogonality established, we may now introduce a definition of inner function that is particular to $\ell^p_A$.  
\begin{Definition}
   Let $1<p<\infty$.  A function $f \in \ell^p_A$ is said to be $p$-inner if it is not identically zero and it satisfies
   \[
          f(z)\ \, \perp_p\ \, z^k f(z)
   \]
   for all positive integers $k$.
\end{Definition}

That is, $f$ is $p$-inner if it is nontrivially orthogonal to all of its forward shifts.  Apart from a  multiplicative constant, this definition, originating from \cite{CMR2}, is equivalent to the traditional meaning of ``inner'' when $p=2$.  Furthermore, this approach to defining an inner property is consistent with that taken in other function spaces \cite{MR1440934, Dima,CMR3,MR1278431,MR1398090,MR1197044,MR936999,Seco}.

Birkhoff-James Orthogonality also plays a role when we utilize a version of the Pythagorean theorem for $\ell^p_A$.
 It takes the form of a family of inequalities relating the lengths of orthogonal vectors with that of their sum \cite[Corollary 3.4]{CR}.

\begin{Theorem}\label{MorePythProp}
Let $1 < r < \infty$, and $1 < p < \infty$.  
   If $1 < p \leq 2 \leq r < \infty$ or $2 \leq p \leq r < \infty$, then there exists $K>0$ such that
  \begin{equation}\label{lowerpyth}   \| f \|^r_p  +  K\|g\|^r_p \leq \|f + g \|^r_p  \end{equation}
  whenever $f \perp_p g$ in $\ell^p_A$.
  
  If $1 < r \leq p \leq 2 $ or $1 < r \leq 2 \leq p < \infty$, then there exists $K>0$ such that
  \begin{equation}   \| f \|^r_p  +  K\|g\|^r_p \geq \|f + g \|^r_p  \end{equation}
  whenever $x \perp_p g$ in $\ell^p_A$.
\end{Theorem}

Actually, this theorem holds with any Lebesgue space $L^p$ in place of $\ell^p_A$.
When $p=2$, the parameters are $K=1$ and $r=2$, and the Pythagorean inequalities reduce to the familiar Pythagorean theorem for a Hilbert space.  More generally, these Pythagorean inequalities enable the application of some Hilbert space methods and techniques to smooth Banach spaces satisfying the weak parallelogram laws; see, for example, \cite[Proposition 4.8.1 and Proposition 4.8.3; Theorem 8.8.1]{CMR}.

For further background on the function space $\ell^p_A$, we refer to the book \cite{CMR}.


\section{Boundedness Conditions}\label{bnd-cond}

Much of the modern treatment of interpolating sequences has been connected to the boundedness of Gramian matrices.  For example, in $H^\infty$, Carleson's theorem tells us that a sequence $\{z_k\} \subset \mathbb{D}$ is interpolating if and only if the corresponding Gramian $G$, given by
\[
G_{i,j} = \langle s_i, s_j \rangle_2,
\]
where $s_i$ is the normalized Szeg\H{o} kernel at $z_i$, is both bounded and bounded below (see \cite[Chapter~9]{AM} for a more general view). However, the boundedness of $G$ is equivalent to the existence of a positive constant $M$ such that
\[
     \Big\|\sum_i a_i s_i \Big\|^2_2 \le M \sum_{i}|a_i|^2,
\]
for any sequence $\{a_i\}$ of scalars. When $p\neq 2$, the conditions in the following propositions, using H\"{o}lder duality, should be read as replacements for this phenomenon. We focus on a direct matrix interpretation of this in Section \ref{matrix-analysis}.

Let us begin by interpolating a fixed finite target when $1<p<\infty$, $p\neq 2$, and letting $1/p + 1/q = 1$. Fixing a positive integer $N$, distinct points $z_0, z_1, z_2,\ldots, z_N \in \mathbb{D}$ and points $w_0, w_1, w_2,\ldots, w_N \in \overline{\mathbb{D}}$, we should first convince ourselves that there exists $h \in \ell^p_A$ such that $h(z_k) = w_k$ for all $k$, $0 \leq k \leq N$.   If we let $B_k$ be the finite Blaschke product with zeros at $\{z_0, z_1, z_2,\ldots, z_N\} \setminus \{z_k\}$, then the function
\begin{equation}  \label{aitchast}
     h_{\ast}(z) = w_0 B_0(z)/B_0(z_0)+ w_1 B_1(z)/B_1(z_1)+ \cdots + w_N B_N(z)/B_N(z_N)
\end{equation}
does the job.  Of course we would also want norm control over such a function.  To obtain it, the following observation is useful.

\begin{Lemma}\label{lookoverthere}
  Let $1<p<\infty$ and $1/p + 1/q = 1$.  Let $\{z_0, z_1,\ldots, z_N\}$ be distinct points in $\mathbb{D}$, and let $\mathscr{J}$ be the subspace of $\ell^p_A$ consisting of functions vanishing on $\{z_0, z_1,\ldots, z_N\}$.  Then 
  \[
        \mathscr{J}^{\perp} =  \bigvee_{\ell^q_A} \{\Lambda_{z_0}, \Lambda_{z_1},\ldots, \Lambda_{z_N}\}.
  \]
\end{Lemma}

\begin{proof} Trivially it holds that
  \[
        \mathscr{J}^{\perp} \supset  \bigvee_{\ell^q_A} \{\Lambda_{z_0}, \Lambda_{z_1},\ldots, \Lambda_{z_N}\}.
  \]
Suppose that $g$ belongs to $\mathscr{J}^{\perp}$, but fails to belong to the subspace $\bigvee_{\ell^q_A} \{\Lambda_{z_0}, \Lambda_{z_1},\ldots, \Lambda_{z_N}\}$; in fact, we may assume that 
 \[
       g \perp_q  \bigvee_{\ell^q_A} \{\Lambda_{z_0}, \Lambda_{z_1},\ldots, \Lambda_{z_N}\}.
  \]
  
  Then, according to \eqref{ppsduwebxzrweq5}, $g^{\langle q-1 \rangle} \in \ell^p_A$ satisfies
  \[
        \langle \Lambda_{z_k}, g^{\langle q-1 \rangle} \rangle = 0
  \]
  for each $k$, $0\leq k\leq N$, which is to say that $g^{\langle q-1 \rangle}$ resides in $\mathscr{J}$.  Finally, this forces
  \[
        \langle g, g^{\langle q-1 \rangle} \rangle = \|g\|_q^q = 0,
  \]
  or $g=0$.
\end{proof}

With that, we have the following interpolation condition for a finite set.

\begin{Proposition}\label{interpropn}
     Let $1<p<\infty$ and $1/p + 1/q = 1$.  Fix a positive integer $N$, distinct points $z_0, z_1,\ldots, z_N \in \mathbb{D}$ and points $w_0, w_1,\ldots, w_N \in {\mathbb{C}}$.  There exists $h \in \ell^p_A$ with $\|h\|_p \leq 1$ such that $h(z_k) = w_k$ for all $k$, $0 \leq k \leq N$, if and only if
     \[
           \sup\left\{ \frac{\big| \beta_0 w_0 + \beta_1 w_1 + \cdots + \beta_N w_N  \big|}
                {\big\| \beta_0 \Lambda_{z_0} +\beta_1 \Lambda_{z_1} +\cdots +\beta_N \Lambda_{z_N} \|_q}  :\
                  (\beta_0, \beta_1,\ldots, \beta_N) \in \mathbb{C}^{N+1}\setminus\{(0,\ldots, 0)\}  \right\}  \ \, \leq \ \, 1.
     \]
\end{Proposition}

\begin{proof}
     Let $\mathscr{J}$ be the subspace of functions in $\ell^p_A$ that vanish at the points $z_0, z_1,\ldots, z_N$ (and possibly elsewhere).  Any member $f$ of $\ell^p_A$ satisfying the conditions $f(z_k) = w_k$, $0 \leq k \leq N$, must be of the form $f = h_{\ast} + J$, where $J \in \mathscr{J}$ and $h_{\ast}$ is as given in \eqref{aitchast}.  The interpolation problem is equivalent to asking whether the coset $h_{\ast} + \mathscr{J}$ contains an element of norm less than or equal to unity.
     
The distance from $h_{\ast}$ to $\mathscr{J}$ is equal to the norm of $h_{\ast}$, viewed as a bounded linear functional on the annihilator of $\mathscr{J}$.  But Lemma \ref{lookoverthere} gives
\[
     \mathscr{J}^{\perp} =  \bigvee_{\ell^q_A} \{\Lambda_{z_0}, \Lambda_{z_1},\ldots, \Lambda_{z_N}\}.
\]

Consequently,
\begin{align*}
         I &=  \inf\big\{ \|h\|_p :\  h(z_k) = w_k,\ 0 \leq k \leq N   \big\} \\
           &=  \inf\big\{ \|h_{\ast} + J \|_p :\ J \in \mathscr{J} \big\} \\
           &=  \sup\left\{ \frac{|\langle h_{\ast}, \psi \rangle|}{\|\psi\|_q}:\  \psi \in \mathscr{J}^{\perp}\setminus\{0\}  \right\} \\
           &=   \sup\left\{ \frac{\big| \beta_0 w_0 + \beta_1 w_1 + \cdots + \beta_N w_N  \big|}
                {\big\| \beta_0 \Lambda_{z_0} +\beta_1 \Lambda_{z_1} +\cdots +\beta_N \Lambda_{z_N} \|_q}  :\
                  (\beta_0, \beta_1,\ldots, \beta_N) \in \mathbb{C}^{N+1}\setminus\{(0, 0,\ldots,0)\} \right\}.
\end{align*}
\end{proof}

This extends to interpolation on an infinite sequence.

\begin{Proposition}\label{interpinfcase}
     Let $1<p<\infty$ and $1/p + 1/q = 1$.  Fix distinct points $z_0, z_1,\ldots \in \mathbb{D}$ and points $w_0, w_1,\ldots \in {\mathbb{C}}$.  There exists $h \in \ell^p_A$ such that $h(z_k) = w_k$ for all $k$, $k \geq 0$ if and only if
     \begin{equation}\label{interpbddcrit}
        \lim_{N\rightarrow\infty}   \sup\left\{ \frac{\big| \beta_0 w_0 + \beta_1 w_1 + \cdots + \beta_N w_N \big|}
                {\big\| \beta_0 \Lambda_{z_0} +\beta_1 \Lambda_{z_1} +\cdots +\beta_N \Lambda_{z_N} \|_q}  :\
                  (\beta_0, \beta_1,\ldots, \beta_N) \in \mathbb{C}^{N+1}\setminus \{{0}\}  \right\}  \ \, < \ \, \infty.
     \end{equation}
\end{Proposition}

\begin{proof}
     Suppose an interpolating function $h$ exists.  By Proposition \ref{interpropn} for every positive integer $N$,
     \begin{align*}
         \|h\|_p  &\geq \inf \big\{\|H\|_p :\ H(z_k) = w_k,\ 0 \leq k \leq N \big\} \\
            &=   \sup\left\{ \frac{\big| \beta_0 w_0 + \beta_1 w_1 + \cdots + \beta_N w_N  \big|}
                {\big\| \beta_0\Lambda_{z_0} +\beta_1 \Lambda_{z_1} +\cdots +\beta_N \Lambda_{z_N} \|_q}  :\
                  (\beta_0, \beta_1,\ldots, \beta_N) \in \mathbb{C}^{N+1}\setminus\{0\}  \right\}.
     \end{align*}
     Then take $N \longrightarrow\infty$.
     
     Conversely, suppose that the bound \eqref{interpbddcrit} holds.  For each $N$, there is the corresponding interpolation function $h_{\ast, N}$ from \eqref{aitchast}, and a function $J = J_N \in \ell^p_A$ for which the extreme value in
     \[
          \inf\big\{ \| h_{\ast,N} + J \|_p: \  J(z_k) = 0, \ 0 \leq k \leq N  \big\}
     \]
     is attained (uniform convexity of $\ell^p_A$ makes this possible).
     
     By the Banach-Alaoglu theorem, there is a weakly convergent subsequence  $\{ h_{\ast,N_k} + J_{N_k}\}_{k=0}^{\infty}$.  Let $h \in \ell^p_A$ be the weak limit (the weak and weak-$\ast$ topologies coincide here since $1<p<\infty$).  By weak convergence, applied to the point evaluations at $z_k$, we see that $h(z_k) = w_k$ for all $k$.  \end{proof}

Our next objective is to interpolate to a target sequence belonging to a particular space, namely $\ell^p$.  Specifically, given $W = (w_k)_{k=0}^{\infty} \in \ell^p$ we wish to find a function $f \in \ell^p_A$ such that 
\[
        (1 - |z_k|^q)^{1/q} f(z_k) = w_k
\]
for all $k$.  The presence of the weight $(1 - |z_k|^q)^{1/q}$ is justified in part by analogous results from other spaces; we shall see that it is a natural choice in the present situation.

\begin{Corollary}\label{suffcondinterp}
    Let $1<p<\infty$, and $1/p + 1/q = 1$.  Suppose that $\{z_0, z_1, z_2,\ldots\}$ is a sequence of distinct points in $\mathbb{D}$ satisfying the condition
    \begin{equation}\label{punifsepcond}
           \sup\left\{ \frac{\big( |c_0|^q/(1 - |z_0|^q)  + |c_1|^q/(1 - |z_1|^q) + \cdots  \big)^{1/q}}
                {\big\| c_0 \Lambda_{z_0} +c_1 \Lambda_{z_1} +\cdots  \|_q}  :\
                  c_k \in \mathbb{C},\ \mbox{not all zero}  \right\}  \  < \ \infty.
    \end{equation}
    Then for every sequence $(w_0, w_1, w_2\ldots) \in \ell^p$ there exists a function $f \in \ell^p_A$ such that
    \[
          (1 - |z_k|^q)^{1/q}  f(z_k) = w_k
    \]
    for all $k$.
\end{Corollary}

\begin{proof} 
    Replace $w_k$ with $w_k/(1 - |z_k|^q)^{1/q}$ in Proposition \ref{interpinfcase}, and apply H\"{o}lder's inequality to get
        \begin{align*}
      &\qquad    \frac{\big| c_0 w_0/(1 - |z_0|^q)^{1/q} + c_1 w_1/(1 - |z_1|^q)^{1/q} + \cdots + c_N w_N/(1 - |z_N|^q)^{1/q}  \big|}
                {\big\| c_0 \Lambda_{z_0} +c_1 \Lambda_{z_1} +\cdots +c_N \Lambda_{z_N} \|_q}  \\
      &\leq \frac{\big( |w_0|^p +\cdots+|w_N|^p  \big)^{1/p} \big(  |c_0|^q/(1 - |z_0|^q) +\cdots + |c_N|^q/(1 - |z_N|^q) + \cdots  \big)^{1/q}}
                {\big\| c_0 \Lambda_{z_0} +c_1 \Lambda_{z_1} +\cdots +c_N \Lambda_{z_N} \|_q}.
         \end{align*}
     Take $N\longrightarrow\infty$.
     \end{proof}

As noted before, $1/(1 - |\zeta|^q)^{1/q} = \|\Lambda_{\zeta}\|_q$ for any $\zeta \in \mathbb{D}$; hence the condition \eqref{punifsepcond} could also be expressed as
\begin{equation}\label{unifsepcond}
      \sup\left\{ \frac{\big( |c_0|^q \|\Lambda_{z_0}\|_q^q  + |c_1|^q \|\Lambda_{z_1}\|_q^q   + \cdots  \big)^{1/q}}
                {\big\| c_0 \Lambda_{z_0} +c_1 \Lambda_{z_1} +\cdots  \|_q}  :\
                  c_k \in \mathbb{C},\ (c_0,c_1,\ldots) \neq 0  \right\}  \  < \ \infty.
\end{equation}

The boundedness conditions on quantities of this type will be important in results to come; let us now give them labels.

\begin{Definition}
Say that a sequence of distinct points $Z = \{z_0, z_1, z_2,\ldots\}$ in $\mathbb{D}$ satisfies the condition (LB) if there exists $C_1 >0$ such that
\[
   \text{(LB)} \hspace{0.5in}        C_1 \leq   \frac{\big\| c_0 \Lambda_{z_0} + c_1 \Lambda_{z_1} +\cdots \big\|_q}{\big(|c_0|^q \|\Lambda_{z_0}\|_q^q + |c_1|^q \|\Lambda_{z_1}\|_q^q +\ldots\big)^{1/q}}  
\]
for all sequences $(c_k)$, not all entries being zero. 

Similarly we say that $Z$ satisfies (UB) if there exists $C_2 >0$ such that
\[
    \text{(UB)} \hspace{0.5in}           \frac{\big\| c_0 \Lambda_{z_0} + c_1 \Lambda_{z_1} +\cdots \big\|_q}{\big(|c_0|^q \|\Lambda_{z_0}\|_q^q + |c_1|^q \|\Lambda_{z_1}\|_q^q +\ldots\big)^{1/q}}    \leq C_2
\]
for all sequences $(c_k)$, not all entries being zero.
\end{Definition}

Notice that both (LB) and (UB) depend on the parameter $q$, the conjugate exponent to $p$, and that we previously encountered (LB) in the form of condition \eqref{unifsepcond}.

Of course, these conditions on a sequence $\{z_0, z_1, z_2,\ldots\}$ could also be viewed as conditions on the corresponding point evaluation kernel functions.  There is existing terminology for this situation. 

\begin{Definition}
Given a sequence of distinct points $Z = \{z_k\}_{k=0}^{\infty}$ of $\mathbb{D}$ we call the set of normalized kernels $\{\Lambda_{z_k}/\|\Lambda_{z_k}\|_q\}_{k=0}^{\infty}$ a {\it Riesz system} in $\ell^q_A$ if both (LB) and (UB) hold for some positive constants $C_1$ and $C_2$.
\end{Definition}

Let us distinguish those sequences on which we are able to interpolate to any target in $\ell^p$.

\begin{Definition}
We call a sequence $Z = \{z_0, z_1, z_2\ldots\}$ of distinct points in $\mathbb{D}$  a {\it universal interpolation sequence} 
for $\ell^p_A$ if for every $\{w_k\}_{k=0}^{\infty} \in \ell^p$ there exists a function $f \in \ell^p_A$ such that $f(z_k)(1 - |z_k|^q)^{1/q} = w_k$ for all $k$.
\end{Definition}

The following two results presented on page 1044 of \cite{VK2} furnish numerous examples of universal interpolation sequences  for $\ell^p_A$.

\begin{Theorem}[Vinogradov] \label{mainvino1} Suppose that $1<p<\infty$.
     Let the sequence $Z = \{z_0, z_1, z_2,\ldots\}$ lie in a Stoltz domain.  Then
     $Z$ is a universal interpolation sequence for $\ell^p_A$ if and only if $Z$ is uniformly separated.
\end{Theorem}

\begin{Theorem}[Vinogradov] \label{mainvino2}
     Suppose that $1<p<\infty$.  If the sequence $Z= \{z_0, z_1, z_2,\ldots\}$ satisfies the conditions
\begin{align*}
     &  |z_k| \leq |z_{k+1}|,\ \ k=0, 1, 2,\ldots;\mbox{\ and}\\
     &  \sup_n(1 - |z_n|)^{-1}\sum_{k\geq n} (1 - |z_k|) < \infty,
\end{align*}
   and there exists $\gamma > 0$ such that 
\[
      \inf_{j\neq k}\Big|\frac{z_j - z_k}{1 - \overline{z}_k z_j}\Big| \geq \gamma,
\]
then $Z$ is a universal interpolation sequence for $\ell^p_A$.
\end{Theorem}

\bigskip

Previously we showed (in Corollary \ref{suffcondinterp}) that a certain  Riesz system for $\ell^q_A$ gives rise to a universal interpolation sequence for $\ell^p_A$.
It turns out that the converse holds, and so we have the following statement.

\begin{Theorem}\label{interprieszbas}
  Let $1<p<\infty$ and $1/p + 1/q =1$.  A sequence $Z = \{z_0, z_1, z_2,\ldots\}$ of distinct points in $\mathbb{D}$ is a universal interpolation sequence for $\ell^p_A$ if and only if $\{\Lambda_{z_k}/\|\Lambda_{z_k}\|_q\}_{k=0}^{\infty}$ is a Riesz system in $\ell^q_A$.
\end{Theorem}

\begin{proof}  It suffices to prove the converse of Corollary \ref{suffcondinterp}.
Define the linear mapping $T$ on $\ell^p_A$ by
\[
       Tf :=  \big(f(z_k)/\|\Lambda_{z_k}\|_q\big)_{k=0}^{\infty}.
\]
  The kernel $\mathscr{K}$ of $T$ is the subspace of functions of $\ell^p_A$ vanishing on $Z$.  To say that $Z$ is a universal interpolating sequence is equivalent to $T$ being surjective from $\ell^p_A$ to $\ell^p$; this is equivalent to the induced mapping
\[
     \tilde{T}: \ell^p_A / \mathscr{K} \longmapsto \ell^p
\]
given by
\[
     \tilde{T}(f + \mathscr{K}) :=  Tf
\]
being invertible.

But the adjoint of $\tilde{T}$ is given by $(a_k) \longmapsto (a_k\Lambda_{z_k}/\|\Lambda_{z_k}\|_q)$, as can be seen by
\begin{align*}
     \langle \tilde{T}(f + \mathscr{K}), (a_k) \rangle &=  \langle (f(z_k)\|\Lambda_{z_k}\|_q), (a_k) \rangle \\
          &= \langle f, (a_k \Lambda_{z_k}/\|\Lambda_{z_k}\|_q) \rangle.
\end{align*}

So boundedness of $\tilde{T}$ and its inverse imply that
\[
    C_1 \leq \frac{ \big\|a_0 \Lambda_{z_0}/\|\Lambda_{z_0}\|_q + a_1 \Lambda_{z_1}/\|\Lambda_{z_1}\|_q +\cdots \big\|_q }{
    \big( |a_0|^q + |a_1|^q + \cdots \big)^{1/q}}  \leq C_2 
\]
for some positive constants $C_1$ and $C_2$.  Now it is a trivial matter to substitute $a_k = c_k\|\Lambda_{z_k}\|_q$ to see that this condition is equivalent to $\{z_k\}$ being a Riesz system in $\ell^q_A$.  
\end{proof}

There is a connection between zero sets of $\ell^p_A$ and interpolation.  
We say that a sequence of distinct points $Z$ in $\mathbb{D}$ is a {\it pre-zero set} of $\ell^p_A$ if there is a nontrivial function in $\ell^p_A$ that vanishes on $Z$, but possibly elsewhere, however.  For $1<p<\infty$, a characterization of the pre-zero sets of $\ell^p_A$, expressed in terms of associated $p$-inner functions, was obtained in \cite{Chengetal1}.  (A precise description of the zero sets is an ongoing challenge.) 

\begin{Proposition}  Let $1<p<\infty$ and $1/p + 1/q = 1$.
     Suppose that $Z = \{z_k\}$ is a sequence of distinct nonzero points in $\mathbb{D}$.  If 
     $Z$ is a universal interpolation sequence for $\ell^p_A$,
     then $Z$ is a pre-zero set for $\ell^p_A$.
\end{Proposition}

\begin{proof}
    By hypothesis there exists $f \in \ell^p_A$ such that $f(z_0) = 1$ and $f(z_k) = 0$ for all $k >0$.  Obviously, this tells us that $\{z_k\}_{k=1}^{\infty}$ is a pre-zero set for $\ell^p_A$.  But then, all of $Z$ is a pre-zero set as well, since  the nontrivial function $(z - z_0)f(z)$ belongs to $\ell^p_A$ and vanishes on $Z$.
\end{proof}


\section{Continuity}\label{cont}

We turn to the problem of interpolating on an infinite sequence in relation to interpolating on a finite truncation of the sequence.  The limiting process is well behaved under certain conditions.

Fix $1<p<\infty$ and let $1/p + 1/q = 1$.  Let $Z = \{z_k\}$ be a sequence of distinct nonzero points in $\mathbb{D}$.  For each $N \in \mathbb{N}$, and each $j$, $0 \leq j \leq N$, let $f_{j,N}$ be the function in $\ell^p_A$ of minimum norm such that 
\begin{align}
    (1 - |z_j|^q)^{1/q}  f_{j,N}(z_j)  &= 1\label{fsubnjay1}\\
    (1 - |z_k|^q)^{1/q}  f_{j,N}(z_k) &= 0,\ \,0\leq k \leq N,\ k \neq j.\label{fsubnjay2}
\end{align}
Such $f_{j,N}$ exists uniquely, since it is the minimum norm element of a nonempty convex set in a uniformly convex space.

Similarly define $f_{j}$ be the function in $\ell^p_A$ of minimum norm such that 
\begin{align}
    (1 - |z_j|^q)^{1/q}  f_{j}(z_j)  &= 1  \label{fsubj1}\\
    (1 - |z_k|^q)^{1/q}  f_{j}(z_k) &= 0,\ \,k \neq j.  \label{fsubj2}
\end{align}
Such $f_j$ need not exist in general, but they do exist of course if $Z$ is a universal interpolation sequence.

We need the following tool to establish the main result of this section.

\begin{Lemma} \label{onemoretool} Let $1<p<\infty$, and $1/p + 1/q = 1$.
   Suppose that $v_1, v_2,\ldots$ are unit vectors in $\ell^p_A$, converging in norm to $v$.  Let $\lambda_n$ be the norming functional of $v_n$, and let $\lambda$ be that of $v$.  Then $\lambda_n$ converges in norm to $\lambda$ in $\ell^q_A$.
\end{Lemma}

\begin{proof}
By definition, all of the norming functionals have unit norm.  Thus by the Banach-Alaoglu theorem, there is a subsequence $\{v_{n_k}\}$ that converges weakly to some $\tilde{\lambda} \in \ell^q_A$ (by reflexivity, the weak and weak* topologies coincide).

If $\tilde{\lambda}$ is normed within $\epsilon$ by the unit vector $\tilde{v}$, then
\[
     \|\tilde{\lambda}\|_{q} \leq \epsilon + \langle \tilde{v}, \tilde{\lambda} \rangle \leq \epsilon + \lim_{k\rightarrow\infty}\big|\langle \tilde{v}, \lambda_{n_k} \rangle\big| \leq \epsilon + \|\tilde{v}\|_{p} \|\lambda_{n_k}\|_{q} = \epsilon + 1.
\]
Consequently, the norm of $\tilde{\lambda}$ is at most one.

Next, for each $k$ we have
\begin{align*}
     \langle v, \tilde{\lambda} \rangle &=   \langle v,  \tilde{\lambda}-\lambda_{n_k}  \rangle +  \langle v, \lambda_{n_k}  \rangle \\
          &=   \langle v,  \tilde{\lambda}-\lambda_{n_k}  \rangle +  \langle v-v_{n_k}, \lambda_{n_k}  \rangle +  \langle v_{n_k}, \lambda_{n_k}  \rangle\\
          &=   \langle v,  \tilde{\lambda}-\lambda_{n_k}  \rangle +  \langle v-v_{n_k}, \lambda_{n_k}  \rangle +  1.
\end{align*}
The first two terms can be made arbitrarily small by choosing $k$ sufficiently large.  This shows that $ \langle v, \tilde{\lambda} \rangle =1$.

But in fact from 
\[
       1 =  \langle v, \tilde{\lambda} \rangle \leq \|v\|_{p} \|\tilde{\lambda}\|_{q} = \|\tilde{\lambda}\|_{q} \leq 1,
\]
we see that $\tilde{\lambda}$ must have unit norm.  This forces $\tilde{\lambda}$ to be a norming functional for $v$.  Since $\ell^p_A$ is smooth, it can only be that $\tilde{\lambda} = \lambda$.  

Finally, from the weak convergence of $\lambda_{n_k}$ to $\lambda$, along with $\|\lambda_{n_k}\|_{q} \rightarrow \|\lambda\|_{q}$, we conclude that $\lambda_{n_k}$ converges to $\lambda$ in norm.  

This argument applies to any subsequence of $\{\lambda_n\}$, with the same limit $\lambda$, and so convergence in $\ell^q_A$ of the sequence itself is assured.
\end{proof}

We are now in a position to show that under broad circumstances, interpolation is well behaved with respect to finite approximation.

\begin{Theorem}\label{contresult}  Let the interpolating functions $f_{j,N}$ be defined as in \eqref{fsubnjay1} and \eqref{fsubnjay2}, and let $f_j$ be defined as in \eqref{fsubj1} and \eqref{fsubj2}.
    For any $j\geq 0$, $f_j$ exists if and only if $\lim_{N\rightarrow\infty}\|f_{j,N}\|_p < \infty$. In this case, $f_{j,N}$ converges in norm to $f_j$ in $\ell^p_A$, as $N$ increases without bound.   
\end{Theorem}

\begin{proof}  To prove this, and establish when the functions $\{f_j\}$ exist, we need the following apparatus from duality.

Fix $j=0$ and $N\geq 1$, and let $\mathscr{M}$ be the subspace of $\ell^p_A$ consisting of functions vanishing on $z_0, z_1,\ldots, z_N$ (and possibly elsewhere).   Then by a basic duality theorem
\begin{align*}
     \|f_{0,N}\|_p &=
     \mbox{dist}(f_{0,N},\mathscr{M}) \\
     &=  \sup \Bigg\{\frac{|\langle f_{0,N}, \psi \rangle |}{\|\psi\|_q}:\ \psi \in \mathscr{M}^{\perp} \Bigg\} \\
          &=  \sup \Bigg\{  \frac{|\langle f_{0,N}, b_0 \Lambda_{z_0} + b_1 \Lambda_{z_1} +\cdots+b_N \Lambda_{z_N} \rangle |}{\|b_0 \Lambda_{z_0} + b_1 \Lambda_{z_1} +\cdots+ b_N\Lambda_{z_N}\|_q} :\  (b_0, b_1,\ldots, b_N) \in \mathbb{C}^{N+1}\setminus \{0\} \Bigg\} \\
          &=   \sup \Bigg\{  \frac{ | b_0 f_{0,N}(z_0) + b_1 f_{0,N}(z_1) +\cdots + b_N f_{0,N}(z_N) |}{\|b_0 \Lambda_{z_0} + b_1 \Lambda_{z_1} +\cdots+b_N \Lambda_{z_N} \|_q} :\  (b_0, b_1,\ldots, b_N) \in \mathbb{C}^{N+1}\setminus \{0\} \Bigg\} \\
          &=   \sup \Bigg\{  \frac{ | b_0|/(1 - |z_0|^q)^{1/q}}{\|b_0 \Lambda_{z_0} + b_1 \Lambda_{z_1} +\cdots+b_N \Lambda_{z_N}\|_q} :\  (b_0, b_1,\ldots, b_N) \in \mathbb{C}^{N+1}\setminus \{0\}\Bigg\} \\
           &=   \sup \Bigg\{  \frac{ 1/(1 - |z_0|^q)^{1/q}}{\|\Lambda_{z_0} + b_1 \Lambda_{z_1} +\cdots+b_N\Lambda_{z_N}\|_q} :\  (b_0, b_1,\ldots, b_N) \in \mathbb{C}^{N+1}\setminus \{0\}\Bigg\} \\
           &=  \big[ \mbox{dist}(\Lambda_{z_0}/\|\Lambda_{z_0}\|_q, \mathscr{N}) \big]^{-1},
\end{align*}
where $\mathscr{N}$ is the span in $\ell^q_A$ of $\{\Lambda_{z_1}, \Lambda_{z_2},\ldots, \Lambda_{z_N}\}$.

Accordingly, let us define $g_{0,N}$ to be the unique metric projection of $\Lambda_{z_0}/\|\Lambda_{z_0}\|_q$ onto the span in $\ell^q_A$ of $\Lambda_{z_1}, \Lambda_{z_2},\ldots, \Lambda_{z_N}$.  We can carry out the analogous exercise to obtain $f_{j,N}$ and $g_{j,N}$ for all $N$ and $0\leq j \leq N$.

By direct calculation we see that $f_{j,N}/\|f_{j,N}\|_p$ is a norming functional for $g_{j,N}$, and likewise 
$g_{j,N}/\|g_{j,N}\|_q$ is a norming functional for $f_{j,N}$.  This further implies that
\begin{equation}\label{normrel1}
       \frac{f_{j,N}^{\langle p-1 \rangle} }{\|f_{j,N}\|_p^{ p-1 }} = \frac{g_{j,N}}{\|g_{j,N}\|_q}.
\end{equation}

Moreover, provided only that $f_j$ exists, we see that $g_j$ is analogously defined, with $f_j/\|f_j\|_p$ and $g_j/\|g_j\|_q$ being mutually norming, and
\begin{equation}\label{normrel2}
       \frac{f_{j}^{\langle p-1 \rangle} }{\|f_{j}\|_p^{ p-1}} = \frac{g_{j}}{\|g_{j}\|_q}
\end{equation}
as well.

Notice that as $N$ increases $g_{j,N}$ is a metric projection onto successively larger subspaces of $\ell^q_A$.
By \cite[Proposition 4.8.3]{CMR}, the metric projection is well behaved in the limit, and we have $g_{j,N}$ convergent in norm to $g_j$ in $\ell^q$.

Apply  Lemma \ref{onemoretool}, in conjunction with the norm relations \eqref{normrel1} and \eqref{normrel2}.  The only way $f_j$ can fail to exist is if $g_{j,N}$ tends to zero.   That implies $\lim_{N\rightarrow\infty}\|f_{j,N}\|_p = \infty$.
\end{proof}

\begin{Corollary}
 Let the functions $f_j$ be defined as in \eqref{fsubj1} and \eqref{fsubj2}.
  For $f_j$ to exist for all $j\geq 0$ it is necessary and sufficient that
  \[
        \mbox{\rm dist}_{\ell^q_A}\big(\Lambda_{z_j}, \mbox{\rm span}\{\Lambda_{z_k}\}_{k\neq j}\big) > 0
  \]
  for all $j$.
\end{Corollary}

The condition in the Corollary is called ``topologically free'' in  \cite{AM}.  It goes by the name ``minimality'' in other works, for example \cite{CMPX}.

The next assertion connects two conditions encountered in our analysis.

\begin{Proposition}
Let $1<p<\infty$, $1/p+1/q =1$, and let $\{z_k\}$ be a sequence of distinct points in the disk $\mathbb{D}$.  If the condition (LB) holds, then so does the separation condition
\[
        \inf \left\{ \left\|\frac{\Lambda_{z_j}}{\| \Lambda_{z_j} \|_q} - \sum_{k\neq j} c_k \frac{\Lambda_{z_k}}{\| \Lambda_{z_k} \|_q}\right\|_q :\ j\geq 0,\ c_k \in \mathbb{C}    \right\}  > 0.
\]
\end{Proposition}

\begin{proof}
      \begin{align*}
          & \inf \left\{ \left\|\frac{\Lambda_{z_j}}{\| \Lambda_{z_j} \|_q} - \sum_{k\neq j} c_k \frac{\Lambda_{z_k}}{\| \Lambda_{z_k} \|_q}\right\|_q :\ j\geq 0,\ c_k \in \mathbb{C}    \right\}\\[.25cm]
               = &    \inf  \left\{ \frac{\left\|c_j\frac{\Lambda_{z_j}}{\| \Lambda_{z_j} \|_q} - \sum_{k\neq j} c_k \frac{\Lambda_{z_k}}{\| \Lambda_{z_k} \|_q}\right\|_q}{|c_j|} :\ j\geq 0,\ c_k \in \mathbb{C},\ c_j \neq 0    \right\}  \\[.25cm]
               \geq &  \inf \left\{ \frac{\left\|c_j\frac{\Lambda_{z_j}}{\| \Lambda_{z_j} \|_q} - \sum_{k\neq j} c_k \frac{\Lambda_{z_k}}{\| \Lambda_{z_k} \|_q}\right\|_q}{\big(\sum_{k} |c_k|^q\big)^{1/q}} :\   c \neq 0    \right\}  \\[.25cm]
      > & 0,
       \end{align*}
       the last inequality being equivalent to (LB).
\end{proof}

\bigskip

Next, we will discuss the convergence of the coefficients of the dual extremal functions. 

For $N \in \mathbb{N}$ and $0 \leq j \leq N$, and with $g_{j,N}$  defined as in the proof of Theorem \ref{contresult}, let us write
\[
     g_{j,N} = \frac{\Lambda_{j}}{\|\Lambda_j\|_q}  -  \sum_{k\neq j} \gamma_{k}^{(j,N)} \frac{\Lambda_{k}}{\|\Lambda_k\|_q},
\]
with the understanding that $\gamma^{(j,N)}_j = 1$ and $\gamma^{(j,N)}_k = 0$ whenever $k>N$.

\begin{Proposition}
     Suppose that $1<p<\infty$, and $1/p+1/q = 1$.  Let a sequence of distinct  points $Z = \{z_0, z_1, z_2,\ldots\}$ be given. 
     If the condition
\[
    M_j :=   \mbox{\rm dist}(\Lambda_{z_j}/\|\Lambda_{z_j}\|_q, \mbox{\rm span}\{\Lambda_{z_k}\}_{k \neq j})   > 0
\]
holds for all $j$, then $\lim_{N\rightarrow\infty} \gamma_k^{(j,N)}$ exists for all $j$ and $k$.
\end{Proposition}

\begin{proof}  The claim is trivial when $j=k$, so we assume otherwise.
Let $r$ and $K$ be suitable Pythagorean parameters for $\ell^q_A$ from Theorem \ref{MorePythProp}, line \eqref{lowerpyth} (see also \cite{CMR}, page 58).  Suppose that $1 \leq M < N$.  Then $g_{j,M} - g_{j,N}$ lies in the span of $\{\Lambda_k:\ k \neq j,\ 0\leq k \leq N\}$.   Thus by the extremal property of $g_{j,N}$,
\[
         g_{j,N}  \perp_q  (g_{j,N} - g_{j,M}).
\]

Consequently, inequality \eqref{lowerpyth} gives
\begin{align*}
     \|g_{j,M}\|^r_q   &\geq  \|g_{j,N}\|^r_q + K\|g_{j,M} - g_{j,N}\|^r_q \\
         &\geq  \|g_{j,N}\|^r_q + K\Big\| \sum_{m\neq j} (\gamma_m^{(j,M)} - \gamma_m^{(j,N)})\Lambda_{z_m}/\| \Lambda_{z_m}\|_q \Big\|^r_q \\
         &\geq  \|g_{j,N}\|^r_q + K |\gamma_k^{(j,M)} - \gamma_k^{(j,N)}|^r \cdot M_k^r.
\end{align*}

We already know that $g_{j,N} \rightarrow g_j$ in norm as $N\rightarrow\infty$.  This forces $\{\gamma_k^{(j,N)}\}_{N=1}^{\infty}$ to be a Cauchy sequence.
\end{proof}

Let us give the name $\gamma^{(j)}_k$ to  $\lim_{N\rightarrow\infty} \gamma_k^{(j,N)}$.   We have thus shown that $g_j$ has the representation
\[
       g_{j}  \sim  \frac{\Lambda_{z_j}}{\|\Lambda_{z_j}\|_q} - \sum_{k\neq j}  \gamma^{(j)}_k \frac{\Lambda_{z_k}}{\|\Lambda_{z_j}\|_q},
\]
with convergence in the sense that the representations for $g_{j,N}$ converge to $g_j$ in $\ell^q_A$.


\section{Matrix Analysis}\label{matrix-analysis}

Next we represent the information obtained thus far about the extremal functions $\{f_j\}$ and their respective dual extremal functions $\{g_j\}$
 in terms of matrices, and uncover another criterion for a universal interpolation sequence for $\ell^p_A$.  We also obtain 
  interesting nonlinear operators relating the associated extremal vectors when $p \neq 2$, extending the notion of a Gramian matrix.

 With the extremal functions $f_j$ as previously defined in \eqref{fsubj1} and \eqref{fsubj2}, let their coefficients by expressed via
 \[
       f_j(z)  :=  \sum_{k=0}^{\infty} \phi_{j,k} z^k.
 \]
 
 Define the matrix 
 \[
      \Phi :=  \left[\begin{array}{cccc} \phi_{0,0}  & \phi_{0,1}  & \phi_{0,2}  & \cdots \\   
           \phi_{1,0}  & \phi_{1,1}  & \phi_{1,2}  & \cdots \\   
             \phi_{2,0}  & \phi_{2,1}  & \phi_{2,2}  & \cdots \\  
             \vdots & \vdots & \vdots & \ddots 
             \end{array}\right]
 \]
 and the matrix $\Psi$ by
 \[
       \big[\Psi\big]_{j,k} :=   z_k^j(1 - |z_k|^q)^{1/q},
\]
where the indices $j$ and $k$ vary from 0 to $\infty$.

Thus the $j$the row of $\Phi$ consists of the coefficients of $f_j$, and the $k$th column of $\Psi$ is populated by the coefficients of the normalized kernel function $\Lambda_{z_k}/\|\Lambda_{z_k}\|_q$.    Notice also that the matrix $\Psi^*\Psi$ is the familiar Gramian matrix in the $p=2$ case, where $\Psi^*$ stands for the conjugate transpose of $\Psi$.  In what follows, we write $\Psi^T$ for the transpose of the matrix $\Psi$, without conjugation.

\begin{Theorem}  Let $1<p<\infty$, and $1/p+1/q =1$.  Suppose that $Z = \{z_j\}$ is a sequence of distinct  points of $\mathbb{D}$.
  If the condition
  \[
        \mbox{\rm dist}(\Lambda_{z_j}/\|\Lambda_{z_j}\|_q, \mbox{\rm span}\{\Lambda_{z_k}\}_{k \neq j})   > 0
  \]
holds for all $j$,  then the following are equivalent.
\begin{itemize}
    \item[(i)]   $Z$ is a universal interpolating sequence.
    \item[(ii)]   The matrix $\Psi$ represents a bounded operator from $\ell^q$ into $\ell^q$ that is bounded below.
    \item[(iii)]   The matrix $\Phi$ represents a bounded invertible operator from $\ell^q$ onto $\ell^q/\mbox{\rm ker} \Psi^T$.
    \end{itemize}
In this case, $I = \Phi \Psi$, and for any $W = \{w_j\}\in \ell^p$ a function $f \in \ell^p_A$ that interpolates to the target $W$ is given by 
$f = \Phi^T W$.
\end{Theorem}

\begin{proof}  The separation condition ensures that each $f_j$ exists, and hence the matrix $\Phi$ exists.  Previously, in Theorem \ref{interprieszbas}, it was established that (i) is equivalent to $\{\Lambda_{z_j}/\|\Lambda_{z_j}\|_q\}_{j=0}^{\infty}$ being a Riesz system; this, in turn, says that $\Phi$ is a bounded, invertible operator from $\ell^q$ to its range in $\ell^q$; that is equivalent to (ii).  If (iii) holds, then for any $W = \{w_j\}\in \ell^p$ a function $f \in \ell^p_A$ that interpolates to the target $W$ is given by 
$f = \Phi^T W$, thus implying (i).  Finally, in any case the matrix equation $I = \Phi \Psi$ holds algebraically.  If (ii) holds as well, then $\Phi^T = (\Psi^T)^{-1}$ (viewed as operating on the range of $\Psi$), and we have (iii).
\end{proof}

This enables us to solve the interpolation problem in terms of $\{f_j\}$, ensuring that the interpolating function can be represented as
\[
    f(z) :=  \sum_{k=0}^{\infty} \frac{f_k(z) w_k}{(1 - |z_k|^q)^{1/q}},
\]
in the sense that the matrix operation supplies the coefficients of $f \in \ell^p_A$.
Vinogradov  \cite{Vino4} obtained similar results under a different set of assumptions on $Z$.

Next, we also obtain interesting non-linear conditions for the associated extremal vectors.

The remainder of this section comprises a set of results that in principle enable us to test an infinite sequence $\{z_0, z_1, z_2,\ldots\}$ for the condition (UB) based on the finite fragments $\{z_0, z_1,\ldots,z_n\}$.

Let distinct  points $\{z_0, z_1,\ldots, z_n\}$ of the disk $\mathbb{D}$ be chosen.  In this situation the matrix $\Psi$ is given by
\begin{equation}\label{ducksauce}
     \Psi = \left[  \begin{array}{cccc}  (1 - |z_0|^q)^{1/q} & (1 - |z_1|^q)^{1/q} & \cdots & (1 - |z_n|^q)^{1/q} \\
               z_0(1 - |z_0|^q)^{1/q} & z_1(1 - |z_1|^q)^{1/q} & \cdots & z_n(1 - |z_n|^q)^{1/q} \\
               z_0^2(1 - |z_0|^q)^{1/q} & z_1^2(1 - |z_1|^q)^{1/q} & \cdots & z_n^2(1 - |z_n|^q)^{1/q} \\  
               \vdots & \vdots & \ddots & \vdots              
               \end{array}      \right].
\end{equation}
We denote its operator norm from $\ell^q$ to $\ell^q(\{0, 1,\ldots,n\})$ by the usual $\|\Psi\|$.

\begin{Lemma}\label{previouslemma}
     Let distinct  points $\{z_0, z_1,\ldots, z_n\}$ of the disk $\mathbb{D}$ be chosen.  With the matrix $\Psi$ defined as in \eqref{ducksauce}, there exist unit vectors $\tilde{a} \in \ell^q(\{0, 1,\ldots,n\})$ and 
     $\tilde{b} \in \ell^q_A$ such that 
     \begin{align*}
         \| \Psi \tilde{a}\|_q  &=  \|\Psi\|\\
         \| \Psi^T \tilde{b}\|_q  &=  \|\Psi\| \\
         \langle \tilde{b},\Psi \tilde{a} \rangle &= \|\Psi\|.
     \end{align*}
\end{Lemma}

\begin{proof}  First, let us observe that the condition (UB) always holds for a finite collection of points $\{z_0, z_1,\ldots, z_n\}$.  This is because
\begin{align*}
      \big\|  a_0 \Lambda_{z_0}/\|\Lambda_{z_0}\|_q  + \cdots +  a_n \Lambda_{z_n}/\|\Lambda_{z_n}\|_q  \big\|_q
     &\leq  \|a\|_q \Bigg[  \Big( \frac{\|\Lambda_{z_0}\|_{q}}{\|\Lambda_{z_0}\|_q} \Big)^{p}+ \cdots +\Big( \frac{\|\Lambda_{z_n}\|_{q}}{\|\Lambda_{z_n}\|_q} \Big)^{p} \Bigg]^{1/p}\\
     &= (n+1)^{1/p}  \|a\|_q.
\end{align*}

Next, for $a = [a_0\ \ a_1\ \ \cdots\ \ a_n]^T$ we can express the condition (UB) as
\[
      \sup_{a \neq 0} \frac{\| \Psi a \|_{q}}{\|a\|_{q}} < \infty.
\]
The supremum in the condition (UB) is attained by some choice of unit vector $\tilde{a}$.   This is because $\|\Psi a\|_q$ is a continuous function of $a$ as $a$ varies varies over the unit sphere of $\mathbb{C}^{n+1}$, a compact set.

With  some selection of $\tilde{a}$ fixed, we obtain $\tilde{b} \in \ell^p$ as follows.   We know that
\[
     \|\Psi\| = \|\Psi \tilde{a}\|_q = \sup\Bigg\{ \frac{|\langle b,\Psi \tilde{a}\rangle|}{\|b\|_p}:\ b \in \ell^p_A,\ b\neq 0  \Bigg\}
\]
We can therefore choose unit vectors $b^{(n)} \in \ell^p$ so that each $\langle b^{(n)},\Psi \tilde{a}\rangle \geq 0$, and
\[
     \|\Psi\| = \|\Psi \tilde{a}\|_q = \lim_{n\rightarrow\infty} \langle b^{(n)},\Psi \tilde{a}\rangle.
\]
By passing to a subsequence, we can assume that the $b^{(n)}$ converge weakly to some ${b} \in \ell^p$.   If $\lambda \in \ell^q$ is the norming functional for ${b}$, then from
\[
     1= \|b^{(n)}\|_p =  \|b^{(n)}\|_p \|\lambda\|_q \geq \big|\langle b^{(n)},\lambda\rangle\big| \rightarrow \langle {b},\lambda \rangle= \|{b}   
     \|,
\]
we see that
\[
     \|\Psi\| = \lim_{n\rightarrow\infty} \langle b^{(n)},\Psi \tilde{a}\rangle \leq \frac{\langle {b},\Psi \tilde{a}\rangle}{\|{b}\|_p} \leq \|\Psi\|.
\]

Equality must hold throughout the previous line.   Put $\tilde{b} = b/\|b\|_p$, and note
\[
   \|\Psi\| = \langle \tilde{b},\Psi\tilde{a}\rangle  =  \langle \Psi^T \tilde{b},\tilde{a}\rangle \leq \|\Psi^T \tilde{b}\|_p \|\tilde{a}\|_q = 
   \frac{ \|\Psi^T \tilde{b}\|_p}{\|\tilde{b}\|_p} \leq \|\Psi\|.
\]
Again, equality prevails, showing that $\|\Psi^T \tilde{b}\|_p = \|\Psi\|$.
\end{proof}

Next we show that the extremal vectors $\tilde{a}$ and $\tilde{b}$ are related by a nonlinear equation when $p\neq 2$.  Recall that we use the notation
\[
       a^{\langle p-1 \rangle} = \big[a_0^{\langle p-1 \rangle} \ \ a_1^{\langle p-1 \rangle} \ \ \ \ \cdots\ \ \ \ a_n^{\langle p-1 \rangle} \big]^T
\]
when
\[
      a  =  \big[ a_0 \ \ a_1 \ \ \ \ \cdots\ \ \ \ a_n \big]^T.
\]

\begin{Theorem}
Let $\{z_0, z_1,\ldots, z_n\}$ be distinct  points of $\mathbb{D}$.  If $\Psi$ is the matrix given by \eqref{ducksauce}, and $\tilde{a}$ and $\tilde{b}$ are associated extremal unit vectors from Lemma \ref{previouslemma}, then 
\[
     \frac{\Psi\tilde{a}}{\|\Psi\|} = {\tilde{b}^{\langle p-1\rangle}} \ \ \ \mbox{and} \ \ \ 
      \frac{\Psi^T \tilde{b}}{\|\Psi\|} = {\tilde{a}^{\langle q-1\rangle}}.
\]
\end{Theorem}

\begin{proof}
In Lemma \eqref{previouslemma}, it was established that
\begin{equation}\label{normedbyab}
           \langle \tilde{b}, \Psi \tilde{a} \rangle  = \|\Psi\|.
\end{equation}

We check by inspection that 
\[
     \big\| \Psi \tilde{a}/\|\Psi\| \big\|_q = 1 \ \ \mbox{and}\ \ \big\|\Psi^T \tilde{b}/\|\Psi\|\big\|_p = 1.
\]

We can therefore read \eqref{normedbyab} as saying that $\Psi \tilde{a}/\|\Psi\|$ is a norming functional for $\tilde{b}$, and $\Psi^T \tilde{b}/\|\Psi\|$ is norming for $\tilde{a}$.

Routine calculation confirms that $\|\tilde{a}^{\langle q -1 \rangle}\|_p = \|\tilde{a}\|_q = 1$, and $\|\tilde{b}^{\langle p-1 \rangle}\|_q = \|\tilde{b}\|_p = 1$.  In addition, we have
\begin{align*}
    \langle \tilde{a}, \tilde{a}^{\langle q -1 \rangle} \rangle &= \|\tilde{a}\|_q^q = 1\\
      \langle \tilde{b}, \tilde{b}^{\langle p -1 \rangle} \rangle &= \|\tilde{b}\|_p^p = 1.\\
\end{align*}

Thus we also find that $\tilde{a}$ is normed by $\tilde{a}^{\langle q -1 \rangle}$, and 
$\tilde{b}$ is normed by $\tilde{b}^{\langle p-1 \rangle}$.

But the spaces $\ell^p$ and $\ell^q$ are smooth, being that $1<p<\infty$, which is to say that norming functionals for these spaces are unique.  Consequently we obtain
\[
     \frac{\Psi\tilde{a}}{\|\Psi\|} = {\tilde{b}^{\langle p-1\rangle}} \ \ \ \mbox{and} \ \ \ 
      \frac{\Psi^T \tilde{b}}{\|\Psi\|} = {\tilde{a}^{\langle q-1\rangle}}.
\]
\end{proof}

By using the property
\[
    (a^{\langle p-1 \rangle})^{\langle q-1 \rangle}  = a
\]
for all $a \in \mathbb{C}$, we immediately derive a fixed-point property for $\tilde{a}$ and $\tilde{b}$.

\begin{Corollary}\label{nonlinearconds}
Let $\{z_0, z_1,\ldots, z_n\}$ be distinct  points of $\mathbb{D}$.  If $\Psi$ is the matrix given by \eqref{ducksauce}, and $\tilde{a}$ and $\tilde{b}$ are associated extremal unit vectors from Lemma \ref{previouslemma}, then
    \begin{align}
     \tilde{a} &=  \frac{\big[\Psi^T(\Psi \tilde{a} )^{\langle q-1\rangle}\big]^{\langle p-1\rangle}}{\|\Psi\|^{p}} \label{fixpta}\\
      \tilde{b} &=  \frac{\big[\Psi(\Psi^T \tilde{b} )^{\langle p-1\rangle}\big]^{\langle q-1\rangle}}{\|\Psi\|^{q}} \label{fixptb}.
  \end{align}
\end{Corollary}

\begin{Remark}
  When $p=2$, the mapping
  \[
        x \longmapsto  \frac{\big[\Psi^T(\Psi x )^{\langle q-1\rangle}\big]^{\langle p-1\rangle}}{\|\Psi\|^{p}} 
  \]
  simplifies to
  \[
       x \longmapsto  \frac{\Psi^*(\Psi x)}{\|\Psi\|^2},
  \]
  which is to say that the nonlinear operator appearing in \eqref{fixpta} generalizes the Gramian matrix when $p \neq 2$, apart from a multiplicative constant.
\end{Remark}

There are converse statements to Corollary \ref{nonlinearconds}, enabling us to retrieve both extremal vectors, if either fixed point condition \eqref{fixpta} or \eqref{fixptb} holds.   Here is the precise statement when a unit vector $\tilde{a}$ satisfies \eqref{fixpta}.

\begin{Theorem}
Let $\{z_0, z_1,\ldots, z_n\}$ be distinct  points of $\mathbb{D}$ and let $\Psi$ be the matrix from \eqref{ducksauce}.  If there exists a unit vector $\tilde{a} \in \ell^q(\{0, 1,\ldots, n\})$ such that
\[
       \tilde{a} =  \frac{\big[\Psi^T(\Psi \tilde{a} )^{\langle q-1\rangle}\big]^{\langle p-1\rangle}}{\|\Psi\|^{p}}, 
\]
and
\[
      \tilde{b} := \Big( \frac{\Psi\tilde{a}}{\|\Psi\|}  \Big)^{\langle q-1 \rangle}
\]
is a unit vector in $\ell^p$,
then $\|\Psi\tilde{a}\|_q = \|\Psi\|$, $\|\Psi^T \tilde{b}\|_p = \|\Psi\|$, $\langle \tilde{b}, \Psi \tilde{a} \rangle  = \|\Psi\|$,
\[
     \frac{\Psi\tilde{a}}{\|\Psi\|} = {\tilde{b}^{\langle p-1\rangle}} \ \ \ \mbox{and} \ \ \ 
      \frac{\Psi^T \tilde{b}}{\|\Psi\|} = {\tilde{a}^{\langle q-1\rangle}},
\]
and $\tilde{b}$ satisfies the fixed point condition \eqref{fixptb}.
\end{Theorem}

\begin{proof}  With $\tilde{b}$ as defined, we immediately have
\[ \frac{\Psi\tilde{a}}{\|\Psi\|} = {\tilde{b}^{\langle p-1\rangle}}  \]
and
\begin{align*}
    \tilde{a} &= \frac{\big[\Psi^T(\Psi \tilde{a} )^{\langle q-1\rangle}\big]^{\langle p-1\rangle}}{\|\Psi\|^{p}}\\
         &=  \frac{\big[\Psi^T(\Psi \tilde{a}/\|\Psi\| )^{\langle q-1\rangle}\big]^{\langle p-1\rangle}}{\|\Psi\|^{p}/\|\Psi\|^{\langle q-1\rangle \langle p-1\rangle}} \\
         &=  \frac{[\Psi^T\tilde{b}]^{\langle p-1\rangle}}{\|\Psi\|^{p-1}}\\
         {\tilde{a}^{\langle q-1\rangle}} &= \frac{\Psi^T \tilde{b}}{\|\Psi\|},
\end{align*}
as claimed.

Next, we obtain the fixed point condition \eqref{fixptb} from
\begin{align*}
    \tilde{a} &=  \frac{[\Psi^T\tilde{b}]^{\langle p-1\rangle}}{\|\Psi\|^{p-1}} \\
    \frac{\Psi \tilde{a}}{\|\Psi\|} &=  \frac{\Psi[\Psi^T\tilde{b}]^{\langle p-1\rangle}}{\|\Psi\|^{p}} \\
    \Big(\frac{\Psi \tilde{a}}{\|\Psi\|}\Big)^{\langle q-1\rangle} &=  \frac{(\Psi[\Psi^T\tilde{b}]^{\langle p-1\rangle})^{\langle q-1\rangle}}{\|\Psi\|^{p(q-1)}} \\
     \tilde{b} &=  \frac{\big[\Psi(\Psi^T \tilde{b} )^{\langle p-1\rangle}\big]^{\langle q-1\rangle}}{\|\Psi\|^{q}}.
\end{align*}

Finally,
\[
     1 = \langle \tilde{a}, (\tilde{a})^{\langle q-1\rangle} \rangle = \frac{\langle \tilde{a}, \Psi^T\tilde{b}\rangle}{\|\Psi\|}
\]
yields $\langle \tilde{a}, \Psi^T\tilde{b}\rangle = \|\Psi\|$, and the inequalities
\begin{align*}
       \|\Psi\| &= \langle \tilde{a}, \Psi^T\tilde{b}\rangle \leq \|\Psi\tilde{a}\|_q \|\tilde{b}\|_p \leq  \|\Psi\| \|\tilde{a}\|_q \|\tilde{b}\|_p = \|\Psi\|\\
        \|\Psi\| &= \langle \tilde{a}, \Psi^T\tilde{b}\rangle \leq \|\tilde{a}\|_q \|\Psi^T \tilde{b}\|_p \leq  \|\Psi\| \|\tilde{a}\|_q \|\tilde{b}\|_p = \|\Psi\|
\end{align*}
force $\|\Psi\tilde{a}\|_q = \|\Psi\|$ and $\|\Psi^T \tilde{b}\|_p = \|\Psi\|$.
\end{proof}

If instead \eqref{fixptb} holds, the statement and proof are analogous.

Finally let us confirm that the extremal property of $\tilde{a}$ is well behaved in the limit.

\begin{Proposition}
    Let $1<p<\infty$ and $1/p + 1/q = 1$.  Suppose that $\{z_0, z_1, z_2,\ldots\}$ is a sequence of distinct  points in the disk $\mathbb{D}$.       If $\Psi$ is the infinite matrix with $(j,k)$th entry $z_k^j(1 - |z_k|^q)^{1/q}$, then either $\Psi$ is a bounded operator on $\ell^q$, and
    \begin{equation}\label{tomatosoup}
           \lim_{n\rightarrow\infty}\Bigg( \sup_{(a_1,\ldots,a_n)\neq 0} \frac{\|\Psi(a_1,\ldots,a_n, 0, 0,\ldots)^T\|_{q}}{\|(a_1,\ldots,a_n)\|_{q}}\Bigg)
              =  \sup_{a \neq 0} \frac{\|\Psi a\|_{q}}{\|a\|_{q}},
    \end{equation}
    or $\Psi$ is unbounded and both sides of \eqref{tomatosoup} are infinite.
\end{Proposition}

\begin{proof}     
First, let us suppose that the infinite matrix $\Psi$ is a bounded operator from $\ell^{q}$ into $\ell^{q}$.    
  For each $N =1, 2, 3,\ldots$ let $\Psi^{(N)}$ be the $\infty\times (N+1)$  matrix consisting of the 0th through $N$th columns of $\Psi$, and let $a^{(N)}$ be an $(N+1)$-dimensional unit vector for which the supremum
\[
      \sup_{a \neq 0} \frac{\|\Psi^{(N)}a\|_q}{\|a\|_q}
\]
is attained.
It is obvious that 
\begin{equation}\label{finiteapprox}
     \lim_{n\rightarrow\infty} \frac{\|\Psi^{(N)}a^{(N)}\|_{q}}{\|a^{(N)}\|_{q}}  \leq \sup_{a\neq 0} \frac{\|\Psi a\|_{q}}{\|a\|_{q}}.
\end{equation}

On the other hand, there exist unit vectors $a^{[n]}$ in $\ell^q$, supported on  the indices $\{0, 1, 2,\ldots, n\}$, such that $\|\Psi a^{[n]}\|_q$ converges to $\|\Psi\|$.

Let $\epsilon > 0$.  For some $N$, we have
\[
    \|\Psi \| - \|\Psi a^{[n]}\|_{q} < \epsilon
\]
whenever $n \geq N$.

Then
\[
   \|\Psi \|-\epsilon  \leq  \frac{\|\Psi a^{[N]}\|_{q}}{\|a^{[N]}\|_{q}}  \leq \frac{\|\Psi^{(N)}a^{(N)}\|_{q}}{\|a^{(N)}\|_{q}}
\]

This shows that equality holds in \eqref{finiteapprox}. 

If $\Psi$ is not a bounded operator on $\ell^q$, then both sides of \eqref{tomatosoup} diverge to infinity.
\end{proof}

These results say that the condition (UB) can be approached by approximating with finite dimensional cases.  Furthermore, they identify the nonlinear relationships between the extremal vectors that apply when $p\neq 2$.


\section{Weak Separation}\label{weak-separation}

Another facet of Carleson's Theorem equates interpolating sequences with a geometric condition and a measure-theoretic condition. Namely, a sequence $Z = \{z_j\}$ is interpolating for $H^\infty$ if and only if both of the following conditions hold:
\begin{enumerate}
\item[(i)] The points of $Z$ are weakly separated in the pseudohyperbolic metric; that is, 
\[
\inf_{j\neq k}\left|\frac{z_j - z_k}{1 - \overline{z_k}z_j} \right| > 0.
\]
\item[(ii)] 
There exists a constant $C > 0$ so that the atomic measure $\mu = \sum_{j\ge 0} (1 - |z_j|^2)\delta_{z_j}$ satisfies
\[
\int_\mathbb{D} |f|^2 \ d\mu \le C \|f\|^2_{H^2}
\]
for all $f\in H^2$.
\end{enumerate}

In our final two sections, we will explore both corresponding conditions for $\ell^p_A$. We begin here by providing a definition for weak separation in the multiplier algebra of $\ell^p_A$, which we will denote by 
\[
\mathscr{M}_p := \left\{\varphi \in \ell^p_A :  \varphi f \in \ell^p_A, \ \forall f\in \ell^p_A\right\}.
\] 
We endow this space with the norm 
\[
\| \varphi\|_{\mathscr{M}_p} = \sup \left\{\| \varphi f \|_p :  f \in \ell^p_A, \ \|f\|_p=1 \right\}.
\]

We point the reader to \cite{MR3714456,CMR} for further background and properties of these spaces.

\begin{Definition}
Let $1 < p < \infty$. A sequence $Z = \{z_0, z_1,\ldots\}\subset \mathbb{D}$ is \textit{weakly separated} by $\mathscr{M}_p$ if there exists $\epsilon > 0$ such that, whenever $i \neq j$, there exists $\phi_{ij} \in \mathscr{M}_p$, with $\|\phi_{ij}\|_{\mathscr{M}_p} \leq1$, that satisfies
\[
\phi_{ij}(z_i) = \epsilon, \ \ \ \phi_{ij}(z_j) = 0. 
\]
\end{Definition}
It is well known that when $p=2$, this definition is equivalent to separation in terms of the pseudohyperbolic metric. We would like to uncover a similar \textit{geometric} condition for $p\neq 2$.

Fix $1<p<\infty$, and let $1/p+1/q=1$.  Let us define 
\[
\rho_p(z_1, z_2)  := \left|\frac{z - z_2}{1 - z_2^{\langle q-1 \rangle}z_1}\right| = \left|\frac{z_1 - z_2}{1 - |z_2|^{q-2}\overline{z_2}z_1}\right|.
\]
We see that when $p=2$, we recover the pseudohyperbolic metric. 
We also note that the function
\[
\varphi(z) = \frac{z - \alpha}{1 - \alpha^{\langle q-1 \rangle}z}
\]
is extremal in the sense that it solves the problem
\[
\sup\left\{|f(0)| : \|f\|_p = 1, f(\alpha) = 0 \right\}.
\]
Again, notice that when $p=2$, we recover the familiar Blaschke factor.  More generally, $\varphi(z)$ is $p$-inner, as discussed in Section \ref{prelim}.  However, products of such functions fail to be $p$-inner when $p \neq 2$; this is one of the ways that the function theory of $\ell^p_A$ is a tough nut to crack.  See \cite{Chengetal1} for details.

It is immediate upon inspection that $\rho_p$ will not be symmetric, i.e. in general $\rho_p(z_1, z_2) \neq \rho_p(z_2, z_1)$. In turn, the best we can hope for is that $\rho_p$ be a \textit{quasimetric}, which turns out to be the case. 

\begin{Proposition}
Let $1 < p < \infty$. The function $\rho_p \colon \mathbb{D}\times\mathbb{D} \to \mathbb{R}$ given by
\[
\rho_p(z_1, z_2)  := \left|\frac{z_1 - z_2}{1 - z_2^{\langle q-1 \rangle}z_1}\right| 
\]
is a quasimetric on the unit disk.
\end{Proposition}
\begin{proof}
The function $\rho_p$ is clearly positive and positive definite. We need now only to prove a triangle inequality. Since the Blaschke factor $\varphi_w$, with zero at $w \in \mathbb{D}$, is an automorphism of $\mathbb{D}$, we have
\[
\rho_p(z_1, z_2) \leq\rho_p(z_1, w) + \rho_p(w, z_2) \iff \rho_p(\varphi_w(z_1), \varphi_w(z_2)) \leq\rho_p(\varphi_w(z_1), 0) + \rho_p(0, \varphi_w(z_2)).
\]
With $\gamma_1 = \varphi_w(z_1)$ and $\gamma_2= \varphi_w(z_2)$, it suffices to show that 
\[
\rho_p(\gamma_1, \gamma_2) \leq\rho_p(\gamma_1, 0) + \rho_p(0, \gamma_2), 
\]
or equivalently that
\[
\left|\frac{\gamma_1 - \gamma_2}{1 - \gamma_2^{\langle q-1 \rangle}\gamma_1}\right| \leq|\gamma_1| + |\gamma_2|.
\]
What we will actually show is that  
\[
\rho_p(\gamma_1, \gamma_2) \leq\frac{|\gamma_1| + |\gamma_2|}{1 + |\gamma_2|^{\langle q-1 \rangle}|\gamma_1|},
\]
which gives the result, as $1 + |\gamma_2|^{\langle q-1 \rangle}|\gamma_1| \geq 1$. 
Notice that 
\begin{align*}
\rho_p(-|\gamma_1|, |\gamma_2|) &= \left|\frac{-|\gamma_1| - |\gamma_2|}{1 - |\gamma_2|^{\langle q-1 \rangle}(-|\gamma_1|)}\right| \\
&=\frac{|\gamma_1| + |\gamma_2|}{1 + |\gamma_2|^{\langle q-1 \rangle}|\gamma_1|} 
\end{align*}
which is the upper bound we would like to establish for $\rho_p(\gamma_1, \gamma_2)$. 
In order to do this, we show that 
\[
1 - \rho_p(\gamma_1, \gamma_2)^2 \geq 1 - \rho_p(-|\gamma_1|, |\gamma_2|)^2.
\]
Suppose $\gamma_1$ and $\gamma_2$ are non-zero (otherwise the inequality is trivial).  Notice that $\Re\left\{\gamma_1\overline{\gamma_2}\right\} \geq -1$ and $|\gamma_1 \gamma_2| \leq1$, which gives 
\[
\frac{\Re\left\{\gamma_1\overline{\gamma_2}\right\}}{|\gamma_1 \gamma_2|} \geq \Re\left\{\gamma_1\overline{\gamma_2}\right\} \geq -1, 
\]
and yields that 
\[
\Re\left\{\gamma_1\overline{\gamma_2}\right\} \geq -|\gamma_1 \gamma_2|.
\]
Let us work backwards from this inequality to get what we want: 
\begin{align*}
&\Re\left\{\gamma_1\overline{\gamma_2}\right\} \geq -|\gamma_1 \gamma_2|\\
\implies & \Re\left\{\gamma_1\overline{\gamma_2}\right\} \geq \left(\frac{|\gamma_2|^{q-2} - 1}{1- |\gamma_2|^{q-2}}\right) |\gamma_1 \gamma_2|\\
\implies & \Re\left\{ \gamma_1 \overline{\gamma_2} \right\} - \Re\left\{\gamma_1\gamma_2^{\langle q-1\rangle} \right\} \geq \left(|\gamma_2|^{q-2} - 1\right) |\gamma_1 \gamma_2|\\
\implies & \Re\left\{ \gamma_1 \overline{\gamma_2} \right\} - \Re\left\{\gamma_1\gamma_2^{\langle q-1\rangle} \right\} \geq 
|\gamma_2|^{\langle q-1\rangle>}|\gamma_1| - |\gamma_1 \gamma_2|\\
\end{align*}
We now multiply each side by $2$ and add $1 + |\gamma_2^{\langle q-1\rangle}\gamma_1|^2 - |\gamma_1|^2 - |\gamma_2|^2$ to the last inequality above:
\begin{align*}
& 2\Re\left\{ \gamma_1 \overline{\gamma_2} \right\} - 2\Re\left\{\gamma_1\gamma_2^{\langle q-1\rangle} \right\} +  1 + |\gamma_2^{\langle q-1\rangle}\gamma_1|^2 - |\gamma_1|^2 - |\gamma_2|^2 \\
& \ \ \ \geq 
2|\gamma_2|^{\langle q-1\rangle}|\gamma_1| - 2|\gamma_1 \gamma_2| + 1 + |\gamma_2^{\langle q-1\rangle}\gamma_1|^2 - |\gamma_1|^2 - |\gamma_2|^2\\
\implies & 
\underbrace{1 - 2\Re\left\{ \gamma_2^{\langle q-1\rangle}\gamma_1\right\} + \left| \gamma_2^{\langle q-1\rangle}\gamma_1 \right|^2}_{|1 - \gamma_2^{\langle q-1\rangle}\gamma_1|^2}
- \underbrace{\left( |\gamma_1|^2 - 2\Re\left\{ \gamma_1 \overline{\gamma_2}\right\} + |\gamma_2|^2 \right)}_{|\gamma_1 - \gamma_2|^2} \\
& \ \ \ \geq 
\underbrace{1 + 2|\gamma_2|^{\langle q-1\rangle}|\gamma_1| + |\gamma_2^{\langle q-1\rangle}\gamma_1|^2}_{(1 + |\gamma_2|^{\langle q-1\rangle}|\gamma_1|)^2}
- \underbrace{\left(|\gamma_1|^2 - 2 |\gamma_1\gamma_2| + |\gamma_2|^2 \right)}_{(|\gamma_1| + |\gamma_1|)^2}.
\end{align*}
Hence, we have 
\[
|1 - \gamma_2^{\langle q-1\rangle}\gamma_1|^2  + |\gamma_1 - \gamma_2|^2 \geq (1 + |\gamma_2|^{\langle q-1\rangle}|\gamma_1|)^2  + (|\gamma_1| + |\gamma_1|)^2. 
\]
This was all to say, using the triangle inequality in the second line below, that
\begin{align*}
1 - \rho_p(\gamma_1, \gamma_2)^2 
& = \frac{|1 - \gamma_2^{\langle q-1\rangle}\gamma_1|^2 - |\gamma_1 - \gamma_2|^2}{|1 - \gamma_2^{\langle q-1\rangle}\gamma_1|^2}\\
&\geq \frac{|1 - \gamma_2^{\langle q-1\rangle}\gamma_1|^2 - |\gamma_1 - \gamma_2|^2}{1 + |\gamma_2^{\langle q-1\rangle}\gamma_1|^2}\\
& \geq \frac{(1 + |\gamma_2|^{\langle q-1\rangle}|\gamma_1|)^2  + (|\gamma_1| + |\gamma_1|)^2}{1 + |\gamma_2^{\langle q-1\rangle}\gamma_1|^2}\\
& = 1 - \rho_p(-|\gamma_1|, |\gamma_2|)^2.
\end{align*}
Thus, we have
\begin{align*}
\rho_p(\gamma_1, \gamma_2) &\leq\rho_p(-|\gamma_1|, |\gamma_2|)\\
&\leq|\gamma_1| + |\gamma_2|\\
&\leq\rho_p(\gamma_1, 0) + \rho_p(0, \gamma_2),
\end{align*}
which establishes the result. 
\end{proof}

In turn, there is some hope that $\rho_p$ may do the job of describing weak separation. That is, for the sequence $Z = \{z_j\}$, we hope to show the separation condition
\begin{equation}\label{weak-sep}
\inf_{n \ne k}  \rho_p(z_n, z_k)  > 0
\end{equation}
is equivalent to $Z$ being weakly separated by $\mathscr{M}_p$. Before doing this, we need a couple of observations, which are variants of the classical Schwarz and Schwarz-Pick lemmata.

\begin{Lemma}\label{p-SL}
Let $1< p < \infty$ and suppose $\phi \in \mathscr{M}_p$, with $\|\phi\|_{\mathscr{M}_p} \leq M$ and $\phi(0)=0$. Then, for all $z\in \mathbb{D}$, 
\[
|\phi(z)| \leq M |z|.
\]

\end{Lemma}
\begin{proof}
The hypotheses ensure that $\phi \in H^\infty$ with $\| \phi \|_{H^\infty} \leq M$ (prop 12.2.4 in \cite{CMR}). This now becomes just a straightforward modification of the classical Schwarz lemma. Let
\[
g(z) =
\begin{cases}
\frac{\phi(z)}{z}, \ z \neq 0\\
\phi'(0), \ z = 0
\end{cases}
\]
which is holomorphic on $\mathbb{D}$. Let $D_r$ be the disk of radius $r$ centered at the origin ($0 < r < 1$). Now, let $z \in D_r$ and invoke the maximum modulus principle to find $z_r \in \partial D_r$ such that
\[
|g(z)| \leq|g(z_r)| = \left| \frac{\phi(z_r)}{z_r} \right| \leq\frac{M}{r}.
\]
As $r \to 1$, we see that $|g(z)| \leq M$, which gives $|\phi(z)| \leq M |z|$.
\end{proof}

We can now prove a replacement of the classical Schwarz-Pick Lemma. 
\begin{Lemma}[$p$-Schwarz-Pick Lemma]\label{p-SPL}
Let $1< p < \infty$. If $f \in \mathscr{M}_p$ with $\|f\|_{\mathscr{M}_p} \le 1$, then, for each $z,w \in \mathbb{D}$, there is a constant $M$ such that 
\[
\rho_p(f(z), f(w)) \leq M \rho_p(z,w) 
\]
\end{Lemma}
Note: the constant $M$ depends on $f$, $p$, and $w$.

\begin{proof}
Let $w \in \mathbb{D}$ and define
\[
\phi_{p,w}(z) := \frac{z - w}{1 - w^{\langle q-1 \rangle}z}.
\]
Note that 
\[
\phi_{p,w}^{-1}(z) = \frac{z + w}{1 + w^{\langle q-1 \rangle}z} = \phi_{p,-w}(z).
\]
Consider the function
\[
\phi_{p, f(w)} \circ f \circ \phi_{p,w}^{-1},
\]
which is defined and bounded on $\mathbb{D}$, and fixes the origin. Thus, we can apply Lemma \ref{p-SL}  to get
\[
\left|\left(\phi_{p, f(w)} \circ f \circ \phi_{p,w}^{-1}\right)(z)\right| \leq M|z|
\]
for all $z\in \mathbb{D}$, with $M = \|\phi_{p, f(w)} \circ f \circ \phi_{p,w}^{-1}\|_\infty$. 
As $\phi_{p,w}^{-1}$ is injective, we have 
\[
\left|\left(\phi_{p, f(w)} \circ f\right)(z) \right| \leq M\left|\phi_{p, w}(z) \right|,
\]
which is precisely
\[
\rho_p(f(z), f(w)) \leq M \rho_p(z,w). 
\]
\end{proof}

We now characterize weakly separated sequences with the following theorem.

\begin{Theorem}
Let $1 < p < \infty$ and let $Z = \{ z_k\}$ be a sequence in $\mathbb{D}$. Then $Z$ is weakly separated by $\mathscr{M}_p$ if and only if 
\[
\inf_{j \ne k} \rho_p(z_j, z_k) > 0.
\]
\end{Theorem}

\begin{proof}
For the forward direction, suppose there exists $\epsilon > 0$ such that, whenever $i \neq j$, there exists $\psi_{ij} \in \mathscr{M}_p$, with $\|\psi_{ij}\|_{\mathscr{M}_p} \leq1$, that satisfies
\[
\psi_{ij}(z_i) = \epsilon, \ \ \ \psi_{ij}(z_j) = 0. 
\]
It suffices to show that there exists $\delta > 0$ such that $\rho_p(z_j, z_k) \geq \delta$ for all $j\neq k$. Now, by the $p$-SPL, we know that for each $z_i,z_j \in Z$, there is a constant $M_{ij}$ such that 
\[
\epsilon = \rho_p(\psi_{ij}(z_i), \psi_{ij}(z_j)) \leq M_{ij} \rho_p(z_i,z_j). 
\]
Here, $M_{ij} = \|\phi_{p, \psi_{ij}(z_j)} \circ \psi_{ij} \circ \phi_{p,z_j}^{-1}\|_{H^{\infty}}= \| \psi_{ij} \circ \phi_{p,z_j}^{-1}\|_{H^{\infty}}$.
If $C = \sup_{i \neq j} M_{ij}$ is finite, then it suffices to take $\delta = \epsilon/C$. But notice that $\phi_{p,z_j}^{-1} (\mathbb{D}) \subseteq \mathbb{D}$, so $\|\phi_{p,z_j}^{-1}\|_{H^\infty} \leq1$. And since the same bound is true for $\psi_{ij}$ by hypothesis, we have $M_{ij} \leq1$, and thus, $C \leq1$. 

For the backward direction, let
\[
\psi_k(w) = \frac{w - z_k}{1 - z_k^{\langle q-1 \rangle}w}
\]
and let
\[
\epsilon = \left( \inf_{j \ne k} \rho_p(z_j, z_k) \right)\left(\sup_{k \geq 1} \|\psi_k \|_1\right)^{-1}
\]
which is positive and finite by assumption. 

Further, let $\hat{\psi}_k = \psi_k/\|\psi_k\|_1$ and let $\hat{\psi}_k(z_j)=\epsilon_{jk}$. Note that 
\[
\left|\epsilon_{jk}\right| = \left|\hat{\psi}_k(z_j)\right| = \left|\frac{z_j - z_k}{1 - z_k^{\langle q-1 \rangle}z_j}\right| \frac{1}{\|\psi_k\|_1} \geq \epsilon
\]
Now take $\phi_{jk} = \frac{\epsilon}{\epsilon_{jk}} \hat{\psi}_k$ and notice that 
\[
\|\phi_{jk}\|_1 = \left|\frac{\epsilon}{\epsilon_{jk}}\right| \|\hat{\psi}_k\|_1 \leq1,
\]
Which means $\phi_{jk}$ is in the closed unit ball of $\mathscr{M}_p$ since $\|\phi_{jk}\|_{\mathscr{M}_p} \leq\|\phi_{jk}\|_1$. Further, whenever $j \neq k$, we have
\[
\phi_{jk}(z_j) = \frac{\epsilon}{\epsilon_{jk}} \hat{\psi}_k(z_j) = \epsilon
\]
and 
\[
\phi_{jk}(z_k) = \frac{\epsilon}{\epsilon_{jk}} \hat{\psi}_k(z_k) = 0.
\]
Thus, $Z$ is weakly separated.

\end{proof}


\section{Carleson Measures}

In this final section we relate the condition (UB) to another class of inequalities, as well as to Carleson measures.  

Recall that a finite measure $\mu$ on $\mathbb{D}$ is said to be a {\it Carleson measure} if for some $C>0$ we have
\[
     \mu(S) \leq Ch
\]
for every set $S$ of the form
\[
      S = \{re^{i\theta} \in \mathbb{D}:\ 1-h\leq r<1,\ \theta_0\leq \theta \leq \theta_0+h\},
\]
where $0 \leq \theta_0 \leq 2\pi$ and $0<h<1$.  That is, $\mu$ is a Carleson measure in terms of  ``Carleson windows.''   

The following characterizes the condition (UB).  It is the dual statement.

\begin{Theorem}\label{carleson2}
  Assume $1<p<\infty$ and $1/p + 1/q=1$, and $\{z_0, z_1, z_2\ldots\}$ is a sequence of distinct points in $\mathbb{D}$.  
 There exists $K>0$ such that
 \begin{equation}\label{carlcondf}
     \sum_{k=0}^{\infty} (1 - |z_k|^q)^{p-1} |f(z_k)|^p  \leq K\|f\|^p_p
\end{equation}
for all $f \in \ell^p_A$,
if and only if (UB) holds, that is, there exists $C>0$ such that
\begin{equation}\label{carlcondbeta}
       \Big\|  \sum_{k=0}^{\infty} c_k \Lambda_{z_k}  \Big\|_q    \leq C  \Big(\sum_{k=0}^{\infty} |c_k|^q \|\Lambda_{z_k}\|^q_q\Big)^{1/q}
\end{equation}
for all constants $\{c_k\}_{k=0}^{\infty}$. 
 \end{Theorem}

\begin{proof}
If \eqref{carlcondf} holds, then for any constants $\{c_k\}_{k=0}^{\infty}$,
\begin{align*}
     \Big\|  \sum_{k=0}^{\infty}c_k \Lambda_{z_k}  \Big\|_q   &=  \sup_{f\neq 0}  \frac{\Big|  \sum_{k=0}^{\infty} c_k \Lambda_{z_k}(f)  \Big|}{\|f\|_p}  \\
       &\leq  \sup_{f \neq 0}  \frac{\Big|  \sum_{k=0}^{\infty} c_k f(z_k)  \Big|}{\|f\|_p} \\
       &\leq   \sup_{f \neq 0}  \frac{ \sum_{k=0}^{\infty} |c_k| |f(z_k)| }{\|f\|_p} \\
       &=  \sup_{f \neq 0}  \frac{ \sum_{k=0}^{\infty} |c_k| \|\Lambda_{z_k}\|_q (1 - |z_k|^q)^{1/q} |f(z_k)| }{\|f\|_p} \\
       &\leq \sup_{f \neq 0} \frac{ \Big(\sum_{k=0}^{\infty} |c_k|^q \|\Lambda_{z_k}\|^q_q\Big)^{1/q} \Big( \sum_{k=0}^{\infty} (1 - |z_k|^q)^{p/q} |f(z_k)|^p\Big)^{1/p} }{\|f\|_p} \\
       &\leq \sup_{f \neq 0}  \frac{ \Big(\sum_{k=0}^{\infty} |c_k|^q \|\Lambda_{z_k}\|^q_q\Big)^{1/q} \big(K\|f\|^p_p \big)^{1/p} }{\|f\|_p} \\
       &=  K^{1/p}\Big(\sum_{k=0}^{\infty} |c_k|^q \|\Lambda_{z_k}\|^q_q\Big)^{1/q},
\end{align*}
where we have used $p/q = p(1 - 1/p) = p-1$.  This shows that
\eqref{carlcondbeta}
holds with $C = K^{1/p}$.

Conversely, suppose we select points $z_0, z_1,\ldots, z_N$ such that 
\begin{equation}\label{unsepcond}
        {\|c_0 \Lambda_{z_0} + c_1 \Lambda_{z_1}+\cdots+c_N \Lambda_{z_N}\|_q} \leq K 
         \Big( |c_0|^q \|\Lambda_{z_0}\|_q^q +\cdots + |c_N|^q \|\Lambda_{z_N}\|_q^q \Big)^{1/q}
\end{equation}
for all $N$ and all constants $\{c_k\}_{k=0}^{N}$.

Then 
\begin{align*}
       \|f\|_p  &\geq \frac{|c_0 f(z_0) + c_1 f(z_1) + \cdots + c_N f(z_N)|}{\|c_0 \Lambda_{z_0} + c_1 \Lambda_{z_1}+\cdots+c_N \Lambda_{z_N}\|_q}\\ 
           &\geq \frac{|c_0 f(z_0) + c_1 f(z_1) + \cdots + c_N f(z_N)|}{K 
         \Big( |c_0|^q \|\Lambda_{z_0}\|_q^q +\cdots + |c_N|^q \|\Lambda_{z_N}\|_q^q \Big)^{1/q}}\\ 
         &= \frac{\big|\beta_0 (1 - |z_0|^q)^{1/q} |f(z_0)| + \beta_1(1 - |z_1|^q)^{1/q}| f(z_1)| + \cdots + \beta_N (1 - |z_N|^q)^{1/q}|f(z_N)|\big|}{K 
         \Big( |\beta_0|^q  + |\beta_1|^q + \cdots + |\beta_N|^q  \Big)^{1/q}},\\ 
\end{align*}
where we have chosen 
\[    \beta_k = \begin{cases} \frac{c_k f(z_k)}{(1 - |z_k|^q)^{1/q} |f(z_k)|}&, \ f(z_k) \neq 0 \\
          0&,\ f(z_k) = 0.\end{cases}
\]

That is, for every  sequence $\beta = \{\beta_k\}_{k=0}^{\infty} \in \ell^q$, we have
\[
         \sum_{k=0}^{\infty} \beta_k (1 - |z_k|^q)^{1/q} |f(z_k) | \leq K \|\beta\|_q \|f\|_p.
\]
Finally, divide by $\|\beta\|_q$ then take the supremum over $\beta \in \ell^q \setminus \{0\}$ to obtain \eqref{carlcondf}.
\end{proof}

\begin{Remark}
In the course of solving the corona problem, Carleson showed \cite{Carl,Carl2} that $\mu$ is a Carleson measure if and only if for some $C>0$,
\begin{equation}\label{altcarl}
      \int_{\mathbb{D}} |f(z)|^p\,d\mu \leq C \|f\|_{H^p}^p
\end{equation}
for all $f$ belonging to the Hardy space $H^p$, where $0<p<\infty$.  For this reason, in some texts the condition \eqref{altcarl} is used as the definition of Carleson measure.   However, we will reserve the name ``Carleson measure'' to those measures satisfying the definition given in the beginning of this section (in terms of ``Carleson windows''), as we shall see that Carleson's result does not carry over from $H^p$ to the spaces $\ell^p_A$ when $2 <p <\infty$.
\end{Remark}

A necessary condition for the points $Z$ to satisfy the condition (UB) follows immediately.

 \begin{Corollary}
   Assume $1<p<\infty$ and $1/p + 1/q=1$.   If there exists $K>0$ such that
\[
       \Big\|  \sum_{k=0}^{\infty} \beta_k \Lambda_{z_k}  \Big\|_q    \leq K  \Big(\sum_{k=0}^{\infty} |\beta_k|^q \|\Lambda_{z_k}\|^q_q\Big)^{1/q}
       \]
 for all $f \in \ell^p_A$, then
 \[
       \sum_{k=0}^{\infty} (1 - |z_k|^q)^{p-1} <\infty.
 \]
 \end{Corollary}
 
 To prove this, use $f(z) = 1$ in the theorem.

Other choices of $f$ in Theorem \ref{carleson2} provide additional necessary conditions on a measure $\mu$ for the norm inequality \eqref{carlcondf} to hold.  This idea is illustrated in the next  proposition.

\begin{Proposition}
Let $1<p<\infty$ and $1/p + 1/q = 1$.  Suppose that the finite measure $\mu$ on $\mathbb{D}$ satisfies
\[
         \int_{\mathbb{D}} |f(z)|^p \,d\mu  \leq C  \|f\|_p^p
\]
for all $f \in \ell^p_A$.  Then there exist constants $A_1$ and $A_2$ such that
\[
           A_1 \Big(|a|^p + \frac{|1-a^2|^p}{1 - |a|^p}\Big) \geq c^p \mu\Big( \Big|\frac{a-z}{1 - \bar{a}z}\Big| \geq c  \Big)
\]
whenever $a \in \mathbb{D}$ and $0<c<1$, and
\[
                A_2 \Big( \frac{1}{1 - |a|^p}  \Big)  \geq c^p \mu\Big( \frac{1}{|1 - az|} \geq c   \Big)
\]
whenever $a \in \mathbb{D}$ and $c>1$.
\end{Proposition}

\begin{proof}
Suppose that $\mu$ satisfies
\[
         \int_{\mathbb{D}} |f(z)|^p \,d\mu  \leq C  \|f\|_p^p
\]
for all $f \in \ell^p_A$.

We use
\[
     \int_{\mathbb{D}} |f(z)|^p \,d\mu  =  \|f\|^p_{L^p(\mu)} 
       = \Bigg[\sup \Big\{ \Big| \int_{\mathbb{D}} f(z)g(z) \,d\mu \Big| :\ g(z) \in L^q(\mu),\ \|g\|_{L^q(\mu)} = 1  \Big\}\Bigg]^p,
\]
along with the choice $g(z) = \overline{f(z)} \chi_E(z)/|f(z)|$, where $E = \{z\in \mathbb{D}:\ |f(z)| \geq c \}$ for $c>0$.  The result is
\begin{align*}
      C\|f\|^p_p  &\geq  \Bigg[  \Big| \int_{\mathbb{D}} f(z)g(z) \,d\mu \Big| \Big/  \Big(\int_{\mathbb{D}} |g(z)|^q\,d\mu\Big)^{1/q}  \Bigg]^p  \\
         &\geq  \big[ c\mu(E) / \mu(E)^{1/q} \big]^p \\
         &= c^p \mu(E).
\end{align*}

By choosing $f(z) = \phi_a(z) = (a-z)/(1 - \bar{a}z)$, a M\"{o}bius transformation, we get bounds on the measures of certain circles and  anti-circles in $\mathbb{D}$.  

Using
\[
   \|\phi\|_p^p =  |a|^p + \frac{|1-a^2|^p}{1 - |a|^p},
\]
we have
\begin{equation}\label{condsonmu}
     C \Big(|a|^p + \frac{|1-a^2|^p}{1 - |a|^p}\Big) \geq c^p \mu\Big( \Big|\frac{a-z}{1 - \bar{a}z}\Big| \geq c  \Big),
\end{equation}

for all $0<c<1$ and $a \in \mathbb{D}$.

We can apply the same idea with $f = \Lambda_a$, with the resulting condition
\begin{equation}\label{samestuntlamb}
    C\Big( \frac{1}{1 - |a|^p}  \Big)  \geq c^p \mu\Big( \frac{1}{|1 - az|} \geq c   \Big).
\end{equation}

\end{proof}

\bigskip

We now connect the inequality   \eqref{carlcondf}  to Carleson measures.  

Suppose that $1<p<2$ and $h>0$ is small.  Choose $c = 1-h$ and $a = h^{1/p}c$.   The region 
\[
          \Big\{z:\ \Big|\frac{a-z}{1 - \bar{a}z}\Big| \geq c \Big\}
\]
appearing in \eqref{condsonmu} is the exterior of a disk contained inside $\mathbb{D}$. Its boundary circle intersects the $x$-axis at the solutions of a quadratic equation.  The region also encloses a Carleson window $S_{\delta}$ of maximal width $\delta$.  From elementary geometry, and taking into account the continuity of the square root function, etc., we can estimate the value of $\delta$ by
\begin{align*}
     \delta &=  1 - \frac{-2a(1-c^2) + \sqrt{4a^2(1-c^2)^2+(1-a^2c^2)(c^2-a^2)}}{-2(1-a^2c^2)} \\
     &= 1 - \frac{-(1-h)h^{1/p}(1-[1-h]^2)+ \kappa}{-(1-h^{2/p}[1-h]^4)},\\
     &\qquad \mbox{where\ }\kappa = \sqrt{h^{2/p}(1-h)^2(1-[1-h]^2)^2+(1-h^{2/p}[1-h]^4)(1-h)^2(1-h^{2/p})},\\
     &\approx 1 - \frac{-2h^{1+1/p} + (1 - h)}{-(1 - h^{2/p})} \\
     &\approx h.
\end{align*}
(Since $1<p<2$, we have that $h$ dominates $h^{2/p}$ in the above estimates.)  The notation $\approx$ means that  the expressions are equal apart from terms of higher than first order in $h$.

And now \eqref{condsonmu} gives
\begin{align*}
     \mu(S_{\delta})  &\leq   \mu\Big( \Big\{z:\ \Big|\frac{a-z}{1 - \bar{a}z}\Big| \geq c \Big\}\Big) \\
       &\leq  \frac{C_1}{c^p} \Big(|a|^p + \frac{|1-a^2|^p}{1 - |a|^p}\Big) \\
       &\leq  2C_1 h
\end{align*}
for $h$ sufficiently small.  

In Figure \ref{carleson-small-p} below, we plot the above inequalities; the most heavily shaded region depicts the Carleson window discussed above.

\begin{center}
\begin{figure}[h!]
\includegraphics{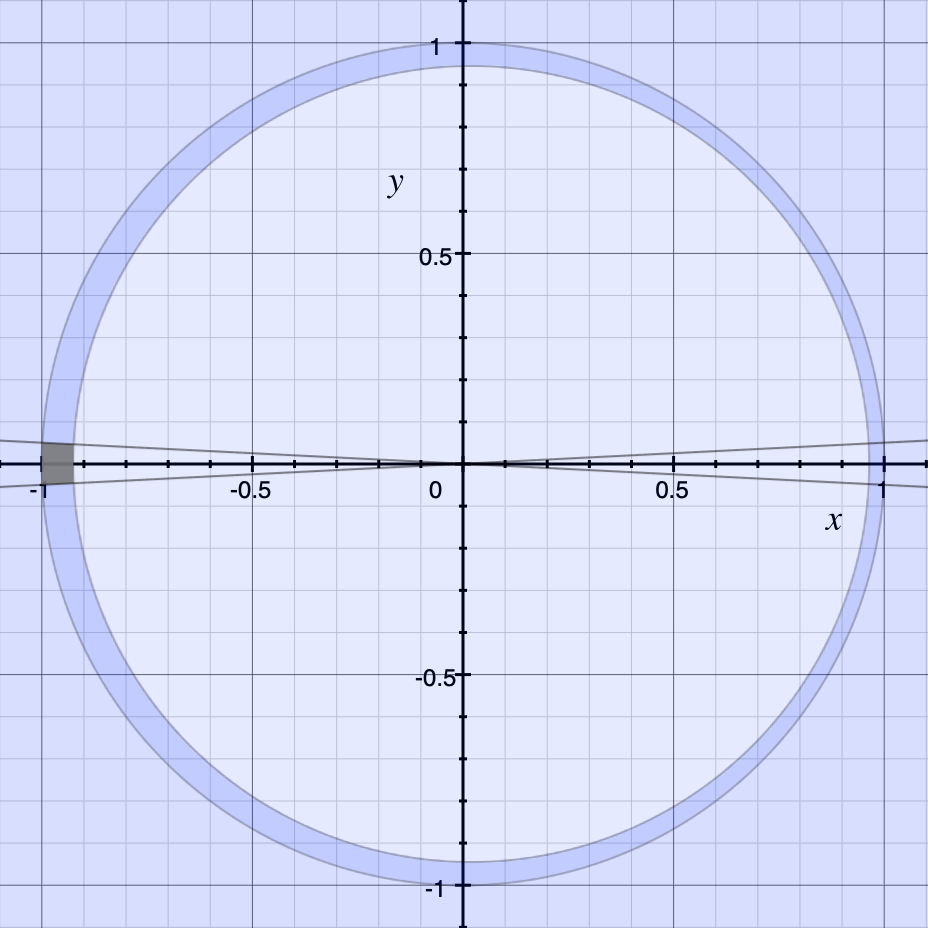}
\caption{The intersection of the unit disk and the region described in \eqref{condsonmu}, for $p = \frac32$ and $h=\frac{1}{10}$, along with a Carleson window.}
\label{carleson-small-p}
\end{figure}
\end{center}

\bigskip

Suppose that $2\leq p<\infty$, and $h>0$ is small.  Choose $c = 1/h$ and $a=(1 - h^{p-1})^{1/p}$.  

The region appearing in \eqref{samestuntlamb} can be described by
\[
         \Big|\frac{1}{a} - z\Big| \leq \frac{1}{ac},
\]
which is the intersection of $\mathbb{D}$ with the disk of radius $1/ac$ centered at $1/a$.  The width of the intersection region is equal to
\[
     \frac{1}{ac} - \Big(\frac{1}{a} - 1\Big) =  \frac{1 - (1-a)c}{ac}  \approx  \frac{h- h^{p-1}/p}{(1-h^{p-1})^{1/p}} \approx h.
\]
(It is in this step that $p\geq 2$ was needed.)

Thus the region contains a Carleson window
\[
     S_{\delta} := \{ re^{i\theta}:\ \ 1-\delta < r < 1;\ |\theta| < \delta/2  \},
\]
where $\delta \approx h$.  

 Figure \ref{carleson-big-p} below shows the Carleson window in this case.

\begin{center}
\begin{figure}[h!]
\includegraphics{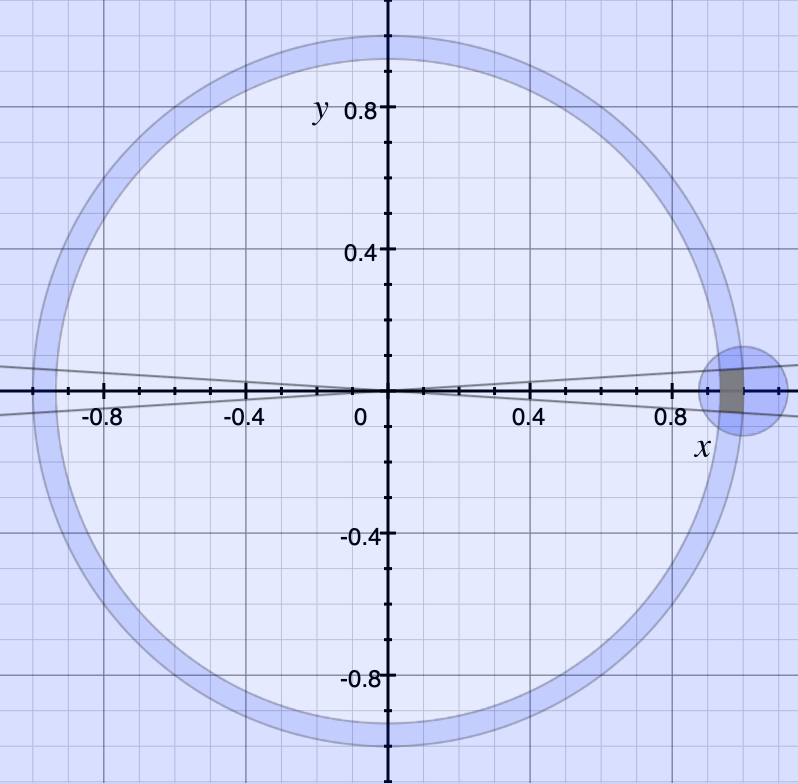}
\caption{The intersection of the unit disk and the region described in \eqref{samestuntlamb}, for $p = 4$ and $h=\frac{1}{8}$, along with a Carleson window (depicted with the heaviest shading).}
\label{carleson-big-p}
\end{figure}
\end{center}

Now \eqref{samestuntlamb} leads to
\begin{align*}
     \mu(S_{\delta})   &\approx     \mu\Big( \frac{1}{|1 - az|} \geq c   \Big) \\
        &\leq    \frac{C_2}{c^p} \Big( \frac{1}{1 - |a|^p}  \Big)  \\
        &=  \frac{C_2 h^p}{h^{p-1}} \\
        &=  C_2 h.
\end{align*}

The argument is rotationally symmetric, and so the following has been proved.  (Of course, the $p=2$ case was already known.)

\begin{Proposition}\label{cheeseboogie}
    Let $1 < p<\infty$.  Suppose that the measure $\mu$ on $\mathbb{D}$ satisfies
\begin{equation}
         \int_{\mathbb{D}} |f(z)|^p \,d\mu  \leq C  \|f\|_p^p
\end{equation}
for all $f \in \ell^p_A$.  Then $\mu$ is a Carleson measure.
\end{Proposition}

Here's part of the converse.  

\begin{Proposition}
    Let $1<p\leq 2$.  If $\mu$ is a Carleson measure, then there exists $C>0$ such that
    \[
           \int_{\mathbb{D}} |f|^p\,d\mu  \leq C \|f\|_{\ell^p_A}^p
    \]
    for all $f \in \ell^p_A$.
\end{Proposition}

\begin{proof}
First, by H\"{o}lder's inequality, we have
\[
       \int_{\mathbb{D}} F\,d\mu  \leq \Big(\int_{\mathbb{D}} F^r\, d\mu \Big)^{1/r} \mu(\mathbb{D})^{1/r'}
\]
for any nonnegative integrable $F$, with $1<r<\infty$ and $1/r + 1/r' = 1$.  Take $1<p<2$ and $r = q/p >1$, where $1/p + 1/q = 1$.  Then for $F = |f|^p$, we have
\[
       \Big( \int_{\mathbb{D}} |f|^p\,d\mu\Big)^{1/p}  \leq \Big(\int_{\mathbb{D}} |f|^q \, d\mu \Big)^{1/q} \mu(\mathbb{D})^{1/(r'p)}.
\]

Thus if $\mu$ is a Carleson measure, and $1<p<2$, then for some $C>0$
\begin{align*}
       \Big( \int_{\mathbb{D}} |f|^p\,d\mu\Big)^{1/p}  &\leq \Big(\int_{\mathbb{D}} |f|^q \, d\mu \Big)^{1/q} \mu(\mathbb{D})^{1/(r'p)} \\
          & \leq  C\mu(\mathbb{D})^{1/(r'p)} \|f\|_{H^q} \\
          & \leq  C\mu(\mathbb{D})^{1/(r'p)} \|f\|_{\ell^p_A}
\end{align*}
for all $f \in \ell^p_A$, with Hausdorff-Young being used in the last step. \end{proof}

\bigskip

The converse to Proposition \ref{cheeseboogie} fails  when $2<p<\infty$, as can be seen from the following example.

\begin{Proposition}
  If $2<p<\infty$, there exists a Carleson measure $\mu$ such that for no $C>0$ does
  \[
          \int_{\mathbb{D}} |f(z)|^p\,d\mu(z) \leq C \|f\|^p_p
\]
hold for all $f \in \ell^p_A$.
\end{Proposition}

\begin{proof}
For $n=1, 2, 3,\ldots$ let $r_n = 1 - 1/n$.  Then of course $\lim_{n\rightarrow\infty} r_n = 1$ monotonically, and 
\[
     \sum_{n=1}^{\infty} (r_{n+1} - r_n) = \sum_{n=1}^{\infty} \frac{1}{n(n+1)} = 1.
\]

Let $\mu$ be the measure defined on $\mathbb{D}$ consisting of point masses at the points $\{r_n\}$ with the respective weights
\[
      \mu(\{r_n\}) = \frac{1}{n(n+1)}.
\]

Then $\mu$ is a Carleson measure, since the amount of mass in a Carleson window of width $1 - r_m$ is at most
\[
     \sum_{n=m}^{\infty} (r_{n+1} - r_n) = 1 - r_m.
\]

Select $\epsilon>0$, and let $2<p<\infty$, and consider the function
\[
    f(z) = \sum_{k=0}^{\infty}  \frac{z^k}{(k+1)^{\epsilon + 1/p}}.
\]
Then $f \in \ell^p_A$.  On the other hand,
\begin{align*}
     \int_{\mathbb{D}} |f(z)|^p\,d\mu(z) &=  \sum_{n=1}^{\infty} \frac{1}{n(n+1)} \Bigg|\sum_{k=0}^{\infty} \frac{(1 - 1/n)^k}{(k+1)^{\epsilon + 1/p}}\Bigg|^p\\
     &\geq \sum_{n=1}^{\infty} \frac{1}{n(n+1)} \Bigg|\sum_{k=0}^{n} \frac{(1 - 1/n)^n}{(k+1)^{\epsilon + 1/p}}\Bigg|^p\\
     &\geq C\sum_{n=1}^{\infty} \frac{1}{n(n+1)} \Big|\frac{1}{e} (n+1)^{1 - \epsilon - 1/p}\Big|^p \\
     &\geq \frac{C}{e^p} \sum_{n=1}^{\infty} (n+1)^{p - 2 - p\epsilon - 1}, 
\end{align*}
where the constant $C$ comes from estimating the sum in $k$ using an integral.

Since $p>2$, we can choose $\epsilon$ sufficiently small that $p-2-\epsilon \geq 0$, so that the final expression diverges to $+\infty$.  This proves the claim.
\end{proof}

It remains, therefore, to fully characterize the norm inequality \eqref{carlcondf} when $2<p<\infty$.  Looking back to earlier sections, we would like to uncover any connection between weak separation and the condition (LB), and find other potential uses for the nonlinear operator extending the Gramian.  These matters are the subject of ongoing work.


\bibliographystyle{plain}

\bibliography{referencesInterpEllp}


\end{document}